\documentclass{article}     

\usepackage[a4paper, outer=3cm]{geometry}
\usepackage{fullpage}
\DeclareOldFontCommand{\bf}{\normalfont\bfseries}{\mathbf}

\usepackage[T1]{fontenc}
\usepackage[utf8]{inputenc}
\usepackage{times}

\usepackage{xspace}

\usepackage{longtable}
\usepackage{amsmath, amsthm, amssymb, amsthm, stmaryrd}
\usepackage{enumerate}
\usepackage{multirow}

\usepackage{algorithm}
\usepackage{algpseudocode}
\algnewcommand\algorithmicinput{\textbf{Input:}}
\algnewcommand\algorithmicoutput{\textbf{Output:}}
\algnewcommand\Input{\item[\algorithmicinput]}
\algnewcommand\Output{\item[\algorithmicoutput]}

\usepackage{hyperref}
\usepackage{cleveref}

\usepackage{pgf,tikz}
\usepackage{svg}
\usetikzlibrary{arrows}
\usetikzlibrary{decorations.markings}
\usepackage[dvipsnames]{xcolor}

\usepackage{pdftricks}
\begin{psinputs}
	\usepackage{pstricks,pst-plot,pst-node,pst-func}
\end{psinputs}
\usepackage{graphics,graphicx}

\tikzstyle{vertex}=[circle, draw, inner sep=1pt, minimum size=6pt]
\newcommand{\vertex}{\node[vertex]}

\usetikzlibrary{decorations.pathreplacing}

\tikzset{->-/.style={decoration={
  markings,
  mark=at position .5 with {\arrow{>}}},postaction={decorate}}}

 \usepackage{marginnote}

\newcommand{\m}[1]{}

\usepackage[nothm]{thmbox}
\usepackage{amsthm}
\usepackage{thmtools}
\usetikzlibrary{patterns}
\usetikzlibrary{calc}

\declaretheorem[parent=section,thmbox=M]{theorem}

\declaretheorem[numberlike=theorem,thmbox=M]{proposition}
\declaretheorem[numberlike=theorem,thmbox=M]{corollary}
\declaretheorem[numberlike=theorem,thmbox=M]{conjecture}
\declaretheorem[numberlike=theorem,thmbox=M]{problem}

\declaretheorem[numberlike=theorem,thmbox=M]{definition}

\newtheorem{claim}{Claim}[theorem]

\declaretheorem[numberlike=theorem]{lemma}

\newcommand{\Gya}{Gy\'arf\'as\xspace}

\newcommand{\Erd}{Erd\H os\xspace}

\newcommand{\sm}{\setminus}

\newcommand{\ra}{\rightarrow}
\newcommand{\Ra}{\Rightarrow}
\newcommand{\la}{\leftarrow}
\newcommand{\La}{\Leftarrow}

\newcommand{\ora}[1]{\overrightarrow{#1}}

\DeclareMathOperator{\chr}{\chi}
\DeclareMathOperator{\ome}{\omega}

\DeclareMathOperator{\dic}{\ora \chi}
\DeclareMathOperator{\diomega}{\ora\omega}

\DeclareMathOperator{\Forb}{Forb}

\newenvironment{proofclaim}
	{\noindent {\bf Proof of Claim
	:}}
	{\hfill $\square$ \par\vspace{11pt}}

\let\le\leqslant
\let\ge\geqslant

\newcommand{\intv}[2]{\llbracket #1 , #2 \rrbracket}
\usepackage{pythonhighlight}

\usepackage{authblk}

\title{Clique Number of Tournaments II}
\author[1]{Guillaume Aubian}
\author[2,3]{Samuel Coulomb}

\affil[1]{Universit\'e Paris-Panth\'eon-Assas, CRED Paris, France.}
\affil[2]{Universit\'e Paris Cit\'e, CNRS, IRIF, F-75013, Paris, France.}
\affil[3]{DIENS, \'Ecole normale sup\'erieure, CNRS, PSL University, Paris, France.}

\begin{document}

\maketitle

\begin{abstract}
    The clique number $\diomega(T)$ of a tournament $T$ is the minimum clique number of a graph formed by ordering the vertices of $T$, removing all arcs going forward with respect to the ordering, and turning the remaining arcs into undirected edges. In this paper, we prove that for every integer $k \ge 3$, deciding $\diomega(T) \le k$ is NP-complete. This answers an interrogation of Nguyen, Scott, and Seymour, and contrasts with the classical undirected setting, where deciding $\omega(G) \le k$ is polynomial-time solvable for all fixed integers $k$. On the other hand, we give a polynomial-time algorithm distinguishing tournaments with $\diomega(T)\le2$ from those with $\diomega(T)>100$.

    We also study the tournament analogue of the \Gya-Sumner conjecture. We construct new $\dic$-bounding tournaments and deduce a conjecture of Aboulker, Aubian, Charbit, and Lopes which states that every class of tournaments with bounded twin-width is $\vec\chi$-bounded. We then exhibit new non $\dic$-bounding tournaments which disprove another conjecture of Aboulker, Aubian, Charbit, and Lopes, as well as two conjectures of Kim. Finally, we present infinite families of $3$-$\diomega$-critical and $4$-$\diomega$-critical tournaments.
\end{abstract}


\tableofcontents 

\section{Introduction}

The \emph{chromatic number} of a graph $G$, denoted $\chr(G)$, is the least integer $k$ such that the vertices of $G$ can be partitioned into $k$ stable sets, and the \emph{clique number} of $G$, denoted $\ome(G)$, is the maximum size of a clique in $G$. These parameters, which are two of Karp's 21 NP-complete problems \cite{Karp}, are among the most studied subjects in graph theory, and have found applications in many other domains.

\medskip

The \emph{dichromatic number} of digraph $D$, denoted $\dic(D)$, is the least integer $k$ such that the vertices of $D$ can be partitioned into $k$ acyclic sets. This parameter generalises the chromatic number to the directed setting. On the other hand, there is no fully consensual definition for the clique number of a digraph (see  \cite{ACN21} and \cite{perfect} for two different generalisations). In 2013, Kim \cite{Kim} proposed a definition that, while less natural at first glance, exhibits good properties. It was later rediscovered in \cite{NSS} and \cite{AACL23}, and has seen more attention since ({\it e.g.}~\cite{pod, Crew, Spirkl, complexity}).

Given a digraph $D$ together with a total ordering $\prec$ of its vertices, we define the backedge graph $D^\prec$ as the undirected ordered graph with vertex set $V(D^\prec) = V(D)$, and an edge $uv$ whenever $vu \in A(D)$ and $u \prec v$, that is, keeping the arcs that go backward with regard to the ordering. It is well-known that a digraph is acyclic if and only if it has an ordering, called a topological ordering, for which every arc goes forward. This is equivalent to having an edgeless backedge graph. With that in mind, it is easy to show that:
\[
    \dic(D) = \min \big\{ \chr(D^\prec) \colon \prec \text{ total ordering of } V(D) \big\}.
\]

This motivates the following definition of the \emph{directed clique number} of a digraph $D$.
\[
    \diomega(D) := \min \big\{ \ome(D^\prec) \colon \prec \text{ total ordering of } V(D) \big\}
\]

The interests of backedge graphs go beyond topological ordering and colouring. Many others digraphs notions may be formulated using this concept, for instance: a feedback vertex set of a digraph is a vertex cover of a backedge graph, a feedback arc set is the edge set of a backedge graph, the acyclic number is the maximum independence number of a backedge graph. One can also define the directed pathwidth \cite{pathwidth}, the degreewidth \cite{degreewidth}, the bandwidth \cite{KL}, and the cutwidth \cite{KL} of a digraph through its backedge graphs.

The study of the directed clique number has so far been limited to the case of tournament, and in this paper, we mostly stay within this framework. A digraph is a tournaments if it has exactly one arc between each pair of vertices, or equivalently, it is an orientation of a complete graph. The class of tournaments is already quite rich (see work such as \cite{Alon} and \cite{heroes}). In particular, every ordered graph is a backedge graph of a (unique) tournament, obtained by replacing each edge by a backward arc, and each non-edge by a forward arc.

\medskip

Computing the clique number of a graph is known to be NP-complete \cite{Karp}. However, for a fixed integer $k \ge 1$, one can decide if a graph on $n$ vertices has clique number at most $k$ in time $O_k(n^{k+1})$, by checking all subsets of $k+1$ vertices. In \cite{NSS}, Nguyen, Scott, and Seymour ask about the complexity of the directed clique number. In Section \ref{sec:NP}, we prove that the problem is NP-complete, even when restricted to tournaments.

\begin{restatable}{theorem}{main}\label{thm:NP}
    For $k \ge 3$, the problem of deciding whether a given tournament $T$ verifies $\diomega(T) \le k$ is NP-complete.
\end{restatable}

When $k=1$, the problem can be solved in polynomial time, as a digraph has clique number at most 1 if and only if it is acyclic. In the case $k=2$, it remains open whether the problem is in P or is NP-hard. Recently, Gutowski and Rams proved that the problem is $\Sigma^P_2$-complete when $k$ is given in the input \cite{complexity}.

On the other side, we investigate approximation algorithms in Section \ref{sec:approx}. We assert that for every $k \ge 0$, there exists an integer $f(k) \ge 0$ and a polynomial-time algorithm that distinguishes between tournaments with $\diomega(T) \le k$ and those with $\diomega(T) > f(k)$. In particular, we describe a polynomial-time algorithm that given a tournament $T$ with $\diomega(T) \le 2$, computes an ordering $\prec$ such that $\omega(T^\prec) \le 100$.

\medskip

The relation between the chromatic number and the clique number of a graph has been extensively studied for decades. While clearly $\ome(G) \le \chr(G)$, one cannot upper bound the chromatic number by a function of the clique number in general. In fact, there are several ways to construct a family of graphs with clique number 2 and unbounded chromatic number \cite{survey}. A natural question is: what substructures must appear in such graphs?

We say that a class of graphs $\cal G$ is \emph{$\chi$-bounded} if there is a function $f$ such that $\chi(G) \le f(\omega(G))$ for all $G \in \cal G$, and we say that a graph $H$ is \emph{$\chi$-bounding} if the class of all graphs not containing $H$ as an induced subgraph is $\chi$-bounded. We refer to \cite{survey} for a survey on $\chi$-boundedness.
In 1959, \Erd \cite{girth} proved that for all $k \ge 0$, there exists graphs with chromatic number at least $k$ and no cycle of length at most $k$. This implies that every $\chi$-bounding graph is a forest. \Gya and Sumner independently conjectured that the converse is true.

\begin{conjecture}[\Gya \cite{Gyarfas} \& Sumner \cite{Sumner}]\label{conj:GS}
    Forests are $\chi$-bounding.
\end{conjecture}

In \cite{AACL23}, Aboulker, Aubian, Charbit, and Lopes defined $\dic$-bounded class of tournaments and $\dic$-bounding tournaments using the directed clique number. They show that every $\dic$-bounding tournaments has a backedge graph which is a forest, and conjecture that the converse holds. In Section \ref{sec:GS}, we disprove this conjecture, and give several counter-examples; these examples also contradict Kim's Conjecture 5.5.10 and Conjecture 5.6.1 in \cite{Kim}. On the other hand, we exhibit new $\dic$-bounding tournaments, and prove that every class of tournaments with bounded twin-width is $\dic$-bounded, confirming another conjecture of \cite{AACL23}.

In Section \ref{sec:crit}, we expose two families, found with Aboulker and Charbit, of 3-$\diomega$-critical and 4-$\diomega$-critical tournaments respectively, and we raise the question of perfect tournaments.

\medskip

Lastly, in Section \ref{sec:open}, we present some open questions on the directed clique number.

\section{Definitions and Notations}

We refer the reader to classical textbooks such as \cite{BondyMurty} for any undefined terminology.

In this paper, an ordering of a graph means a total ordering of its vertices.
Given an integer $n \ge 0$, we let $[n]$ denote the set $\{1, \dots, n\}$, and for a set $X$, we denote by $\binom{X}{n}$ the set of all subsets of $X$ of size exactly $n$.

\paragraph{Ordered graphs}
An \emph{ordered graph} is a graph $G$ with an ordering $<_G$. 
Let $G$ and $H$ be two ordered graphs.
We say that $H$ is an induced ordered subgraph of $G$ if there is an injective map $\varphi \colon V(H) \ra V(G)$ such that for all $u,v \in V(H)$, it holds that $u <_H v$ if and only if $\varphi(u) <_G \varphi(v)$, and $\varphi(u)\varphi(v) \in E(G)$ if and only if $uv \in E(H)$.
We say that $G$ is \emph{$H$-free} if $H$ is not an induced ordered subgraph of $G$, and we denote by $\Forb(H)$ the class of all $H$-free ordered graphs.

The chromatic number and the clique number of an ordered graph $G$ are that of the underlying unordered graph.
A class of ordered graphs $\cal G$ is \emph{$\chr$-bounded} if there exists a function $f$ such that $\chr(G) \le f(\ome(G))$ for all $G \in \cal G$, and we say that $H$ is \emph{$\chr$-bounding} if $\Forb(H)$ is $\chr$-bounded.

\paragraph{Digraphs}
Let $D$ be a digraph. If $xy \in A(D)$, we say that $x$ is an in \emph{in-neighbour} of $y$, and $y$ an \emph{out-neighbour} of $x$. Given a vertex $x$, we let $N^+(x)$ and $N^-(x)$ denote respectively the set of out-neighbours and the set in-neighbours of $x$.
For two disjoint sets of vertices $X$ and $Y$, we write $X \Rightarrow Y$ to say that $xy \in A(D)$ for all $x \in X$ and all $y \in Y$, and we write $X \rightarrow Y$ to say that there is no arc from $Y$ to $X$. For simplicity ,when $X = \{x\}$ is a singleton, we write $x \Ra Y$ and $x \ra Y$ in place of $\{x\} \Ra Y$ and $\{x\} \ra Y$ respectively.

We also use the symbol $\Ra$ to denote a composition  operation on digraphs: given two digraphs $D_1$ and $D_2$, we let $D_1\Ra D_2$ denote the digraph obtained from the disjoint union of $D_1$ and $D_2$ by adding all arcs from $V(D_1)$ to $V(D_2)$. Lastly, we denote by $D[X]$ the subdigraph of $D$ induced by $X$.

\paragraph{Backedge graphs}
Let $D$ be a digraph and $\prec$ an ordering of $D$. Given two disjoint subset of vertices $A,B \subseteq V(D)$, we write $A \prec B$ to say that $a \prec b$ for all $a \in A$ and all $b \in B$. An arc $uv \in A(D)$ is called \emph{forward} if $u \prec v$, and \emph{backward} if $v \prec u$. The \emph{backedge graph} $D^{\prec}$ of $D$ with respect to $\prec$ is the undirected ordered graph on the vertex set $V(D)$ with an edge $uv$ whenever there is a backward arc between $u$ and $v$.

\paragraph{Dicolouring and directed clique number}
Let $D$ be a digraph and $k \ge 0$ an integer. A \emph{$k$-dicolouring} of $D$ is a partition of its vertices into $k$ subsets each inducing an acyclic subdigraph. Alternatively, it is a function  $\phi \colon V(D) \rightarrow [k]$ such that $D[\phi^{-1}(c)]$ is acyclic for each colour $c\in [k]$. The \emph{dichromatic number} of $D$, denoted by $\vec{\chi}(D)$, is the least integer $k$ such that $D$ has a $k$-dicolouring. It is folklore that
\begin{equation}\label{eq:dic}
    \dic(D) = \min \big\{ \chr(D^\prec) \colon \prec \text{ ordering of } D \big\}.
\end{equation}
We define the \emph{directed clique number} $\diomega(D)$ of $D$ as :
\[
    \diomega(D) = \min \big\{ \ome(D^\prec) \colon \prec \text{ ordering of } D \big\}.
\]
For a subset of vertices $X \subseteq V(D)$, when $D$ is clear from the context, we simply write $\dic(X)$ and $\diomega(X)$ in place of $\dic(D[X])$ and $\diomega(D[X])$ respectively.
An ordering $\prec$ of $D$ is called an $\diomega$-ordering if $\omega(D^{\prec})= \diomega(D)$, and a $\dic$-ordering if $\chr(D^{\prec}) = \dic(D)$.

\paragraph{Tournaments}
A \emph{tournament} is an orientation of a complete graph, that is, an oriented graph with exactly one arc between every pair of vertices. We say that a tournament is \emph{transitive} it it is acyclic, and we denote by $TT_n$ the unique transitive tournament on $n$ vertices.
Given two tournaments $T$ and $H$, we say that $T$ is \emph{$H$-free} if it $T$ contains no subtournament isomorphic to $H$, and we denote by $\Forb(H)$ the class of all $H$-free tournaments. A class of tournaments $\cal T$ is \emph{$\dic$-bounded} if there is a function $f$ such that $\dic(T) \le f(\diomega(T))$ for every tournament $T \in \cal T$, and we say that a tournament $H$ is \emph{$\dic$-bounding} if $\Forb(H)$ is $\dic$-bounded.

Given three tournaments $T_1, T_2, T_3$, we denote by $\Delta(T_1,T_2,T_3)$ the tournament obtained from  disjoint copies of $T_1, T_2, T_3$ by adding arcs between them so that $T_1 \Ra T_2 \Ra T_3 \Ra T_1$. For the sake of clarity, in this last notation, we may replace the tournament $TT_k$ with the integer $k$; for example, $\Delta(1,k,T)$ stands for $\Delta(TT_1,TT_k,T)$.

\section{Computing the Clique Number of Tournaments is NP-complete}\label{sec:NP}

This section is about computing the clique number of tournaments. Our main result is the following:

\main*

For $k=1$, this problem is in P, and the case $k=2$ remains open. Theorem \ref{thm:NP} directly implies:

\begin{theorem}\label{thm:diomega_NP}
    Computing the clique number of tournaments is NP-hard.
\end{theorem}

Note that while Theorem~\ref{thm:diomega_NP} is expected, as an analogue of the undirected case, Theorem~\ref{thm:NP} is surprising, in that it contrasts with the undirected case where the corresponding decision problem can be answered in polynomial time. To prove this, we will need intermediate results: this is the purpose of Subsection~\ref{subsec:tools}.

\subsection{Useful tools to increase the clique number}\label{subsec:tools}

The construction in the following proof will be used to assert that Definition~\ref{def:T_k} is well-founded.

\begin{lemma}\label{lem:cut}
    For every tournament $T$, there exists a tournament $T'$ with $\diomega(T') = \diomega(T)$ and such that for all subset of vertices $X \subseteq V(T')$, either $T'[X]$ contains a copy of $T$, or $T'[\overline{X}]$ contains a copy of $T$.
\end{lemma}

\begin{proof}
    Denote $n = |V(T)|$ and $m = {{n(n-1)+1} \choose n}$. Let $T_1, \dots, T_{m \times n}$ be $m \times n$ copies of $T$, and for $i \in [n]$, let $B_i = \{(i-1)m+1, \dots, i m\}$. Note that $(B_i)_{i \in [n]}$ is a partition of $[m \times n]$, thus we can map each $x \in [m \times n]$ to an index $B^{-1}(x) \in [n]$ such that $x \in B_{B^{-1}(x)}$.
    For every $i \in [n]$, let $\varphi_i$ be a bijection from $B_i$ to ${[n(n-1) + 1]} \choose n$, and for each $j \in B_i$, let $\psi_j$ be a bijection from $V(T_j)$ to $\varphi_i(j)$.
    Thus, we associate to every $j \in B_i$ a list of $n$ integers $\varphi_i(j)$, and to each vertex $v \in V(T_j)$ an element in this list.
    Fix a $\diomega$-ordering $v_1, \dots, v_n$ of $T$. Consider the tournament $T'$ obtained from $T_1 \Ra \cdots \Ra T_{m \times n}$ by reversing an arc $uw$ with $u \in T_i$ and $w \in T_j$ if and only if:
    \begin{itemize}
        \item $i < j$,
        \item $\psi_i(u) = \psi_j(w)$, and
        \item $v_{B^{-1}(j)} v_{B^{-1}(i)} \in A(T)$.
    \end{itemize}

    Let us first prove that $\diomega(T') = \diomega(T)$. Let $\prec$ be the ordering of $T'$ such that if $u$ is the copy of vertex $v_i$ in $T_j$ and $w$ is the copy of vertex $v_{i'}$ in $T_{j'}$, then $u \prec w$ if and only if $(j,i)$ is lexicographically smaller than $(j',i')$. Let $K \subseteq V(T')$ be a clique in $T'^{\prec}$.
    If for every $i \in [n]$, $K$ intersects $\bigcup_{j \in B_i} V(T_j)$ on at most one vertex $u_i$, then the set $\{v_i \mid u_i \in K\}$ induces a clique in the fixed $\diomega$-ordering of $T$, and thus $|K| \le \diomega(T)$.
    
    Otherwise $|K \cap \bigcup_{j \in B_i} V(T_j)| \ge 2$ for some $i \in [n]$. Denote $X = K \cap \bigcup_{j \in B_i} V(T_j)$ and let $u,w \in X$ be distinct vertices with $u \in V(T_j)$ and $w \in V(T_k)$. We cannot have $k \ne j$, since there would only forward arcs between $T_i$ and $T_j$ in $T^\prec$. Hence $k = j$, and thus $\psi_j(u) \neq \psi_j(w)$. We cannot have $y \in K \setminus X$ as this implies $\psi_j(u) =\psi_\ell(y) = \psi_j(w)$. It follows that $K \subseteq T_k$. Since $\prec_{\mid V(T_j)}$ is a $\diomega$-ordering of $T_j$, we have $|K| \le \diomega(T)$. Therefore $\diomega(T') \le \diomega(T)$. Clearly $\diomega(T') \ge \diomega(T)$ as $T$ is a subtournament of $T'$. 

    Let us now prove by contradiction that for all $X \subseteq V(T')$, either $T'[X]$ or $T'[\overline{X}]$ contains a copy of $T$. Suppose this is not the case for some set $X \subseteq V(T')$. Then, for every $i \in [m \times n]$, there exists a vertex $u_i \in X \cap T_i$, as otherwise $T_i \subseteq \overline X$.
    For each $j \in [n]$, let $C_j = \{\psi_i(u_i) \mid i \in B_j\}$. If $|\overline C_j| \ge n$, let $C \subseteq \overline C_j$ be a set of size $n$ and let $i = \varphi^{-1}(C)$, then $\psi_i(u_i) \in C \subseteq \overline C_j$, contradicting that $\psi_i(u_i) \in C_j$. Hence $|\overline C_j| \le n-1$ and $|\bigcup_{j \in [n]} \overline C_j| \le n(n-1)$, so there exists $c \in \bigcap_{j \in [n]} C_j$.
    Thus, for every $j \in [n]$, there exists a vertex $w_j \in T_i \cap X$ with $i \in B_j$ such that $\psi_i(w_j) = c$.
    By definition of $T'$,  we have $w_j w_j' \in A(T')$ if and only if $v_j v_j' \in A(T)$, so $T'[\{w_j\}_{j \in [n]}]$ is isomorphic to $T$, yet is included in $X$, a contradiction.
\end{proof}

Using this first lemma, we define the tournament $T_k$ as follows:

\begin{definition}\label{def:T_k}
    Throughout this section, for every interger $k \ge 1$, we denote by $T_k$ a tournament with $\diomega(T_k) = k$ and such that for all subset of vertices $X \subseteq V(T_k)$, either $\diomega(X)=k$ or $\diomega(\overline{X})=k$.
\end{definition}

Note that such a tournament necessarily exists due to Lemma~\ref{lem:cut}, and the fact that there are tournaments with arbitrarily large clique number \cite{AACL23}. The first use of $T_k$, is to force an order on some pairs of vertices.

\begin{lemma}\label{lem:complete}
    Let $k \ge 1$ be an integer, $D$ a digraph, and $u,v \in V(D)$ two vertices such that $N^+(u) \cap N^-(v)$ contains a copy of $T_k$. Then, every ordering $\prec$ of $D$ with $\omega(D^\prec) \le k$ verifies $u \prec v$.
\end{lemma}

\begin{proof}
    Denote $W^\prec = \{w \in N^+(u) \cap N^-(v) \mid w \prec u\}$ and $W^\succ = \{w \in N^+(u) \cap N^-(v) \mid u \prec w\}$. Suppose toward a contradiction that $v \prec u$. Since $u \Ra W$, all vertices of $W^\prec$ are adjacent to $u$ in $D^\prec$, so $\diomega(W^\prec) \le k - 1$. Similarly, using that $v \prec u \prec W^\succ$ and $v \La W$, we have $\diomega(W^\prec) \le k - 1$. This implies that $T_k$ can be partitioned into two set with clique number at most $k-1$, a contradiction.
\end{proof}

We can also use $T_k$ to leverage $\diomega$-orderings of a tournament of a given clique number into $\diomega$-orderings of a tournament with a larger clique number.

\begin{lemma}\label{lem:uplift}
    Let $k \ge 2$ be an integer and $D$ a digraph with $\diomega(D) = k-1$. Then, $\diomega(\Delta(1, D, T_k)) = k$ and 
    \[
        \{\prec \mid \omega(D^\prec) = k-1\} = \{\prec_{\mid V(D)} \mid \omega(\Delta(1, D, T_k)^\prec) = k\}.
    \]
\end{lemma}

\begin{proof}
    Denote $D' = \Delta(1, D, T_k)$. Clearly $\diomega(D') \ge \diomega(T_k) = k$ since $T_k$ is a subdigraph of $D'$.
    
    Denote $v \in V(D')$ the unique vertex of $D'$ that is neither in $D$ nor in $T_k$. Let $\prec_k$ and $\prec$ be $\diomega$-orderings of $T_k$ and $D$ respectively. Consider the ordering $\prec'$ of $D'$ obtained by concatenating the vertices of $D$ ordered by $\prec$, followed by the vertices of $T_k$ ordered by $\prec_k$, and then the vertex $v$. It holds that $\omega(D'^{\prec'}) = k$ and $\prec'_{\mid V(D)} = \prec$. Therefore $\diomega(D') \le k$ and $\{\prec \mid \omega(D^\prec) = k-1\} \subseteq \{\prec_{\mid V(D)} \mid \omega(D'^\prec) = k\}$.
    
    Let $\prec'$ be an ordering of $D'$ such that $\omega(D'^\prec) = k$. For all $u \in V(D)$, as $N^+(u) \cap N^-(v)$ contains a copy of $T_k$, we have $u \prec' v$ by Lemma~\ref{lem:complete}. Hence $D \prec' v$. Moreover $v \Ra D$, so all vertices of $D$ are adjacent to $v$ in $D'$, thus $\omega(D^{\prec'}_k) = k - 1$. If follows that $\{\prec_{\mid V(D)} \mid \omega(D'^\prec) = k\} \subseteq \{\prec \mid \omega(D^\prec) = k-1\}$.
\end{proof}

\subsection{Gadgets and NP-completeness}

We prove Theorem~\ref{thm:NP} via a reduction from $3$-SAT. We now describe gadgets that will be used in that reduction. We first build a tournament in which one fixed edge is forward if and only if another fixed edge is not. This tournament will encode a variable, and which edge is forward will encode whether this variable is set to true.

\begin{lemma}\label{lem:T_var}
    There exists a tournament $T_{var}$ with $\diomega(T_{var}) = 3$ and two disjoint arcs $uv$ and $wx$ such that:
\begin{itemize}
    \item in every $\diomega$-ordering of $T_{var}$, exactly one of $uv$ or $wx$ is forward,
    \item there exists a $\diomega$-ordering in which $uv$ is forward, and
    \item there exists a $\diomega$-ordering in which $uv$ is backward.
\end{itemize}
\end{lemma}

This proof relies on computer help. Since checking all orderings of a graph is a tedious task, we proceed by first finding such a tournament $T$ with $\diomega(T) = 2$, and then leveraging it using Lemma~\ref{lem:uplift}. 

\begin{proof}
    Let $T$ be the tournament with the following ordered graph as a backedge graph:

    \begin{figure}[H]
        \centering
         \begin{tikzpicture}
             \vertex (8) at (1,0) {$1$};
             \vertex (2) at (2,0) {$2$};
             \vertex (1) at (3,0) {$3$};
             \vertex (5) at (4,0) {$4$};
             \vertex (3) at (5,0) {$5$};
             \vertex (4) at (6,0) {$6$};
             \vertex (0) at (7,0) {$7$};
             \vertex (7) at (8,0) {$8$};
             \vertex (6) at (9,0) {$9$};
    
             \draw[bend right=30] (8) to (0);
             \draw[bend right=30] (8) to (7);
             \draw[bend right=30] (8) to (6);
             \draw[bend right=30] (2) to (3);
             \draw[bend right=30] (2) to (4);
             \draw[bend left=30] (2) to (7);
             \draw[bend left=30] (1) to (7);
             \draw[bend left=30] (1) to (6);
             \draw[bend left=30] (5) to (4);
             \draw[bend left=30] (3) to (6);
             \draw[bend left=30] (4) to (6);
        \end{tikzpicture}
        \caption{The backedge graph corresponding to the ordering $1 \prec 2 \prec 3 \prec 4 \prec 5 \prec 6 \prec 7 \prec 8 \prec 9$}
        \label{fig:not_first}
    \end{figure}
    
    Note that $T$ is not transitive and the ordering $1 \prec 2 \prec 3 \prec 4 \prec 5 \prec 6 \prec 7 \prec 8 \prec 9$ is such that $\omega(T^\prec) = 2$, thus $\diomega(T) = 2$. Let $uv = 7 \rightarrow 9$ and $wx = 8 \rightarrow 3$.
    
    The ordering $1 \prec 2 \prec 3 \prec 4 \prec 5 \prec 6 \prec 7 \prec 8 \prec 9$ is a $\diomega$-ordering in which $uv$ is forward. See Figure~\ref{fig:not_first}.

    The ordering $6 \prec 8 \prec 2 \prec 9 \prec 1 \prec 3 \prec 4 \prec 5 \prec 7$ is a $\diomega$-ordering in which $uv$ is backward. See Figure~\ref{fig:not_second}.

    \begin{figure}[H]
        \centering
        \begin{tikzpicture}
        
             \vertex (4) at (1,0) {$6$};
             \vertex (7) at (2,0) {$8$};
             \vertex (2) at (3,0) {$2$};
             \vertex (6) at (4,0) {$9$};
             \vertex (8) at (5,0) {$1$};
             \vertex (1) at (6,0) {$3$};
             \vertex (5) at (7,0) {$4$};
             \vertex (3) at (8,0) {$5$};
             \vertex (0) at (9,0) {$7$};
    
             \draw[bend right=30] (4) to (6);
             \draw[bend right=30] (4) to (8);
             \draw[bend right=30] (4) to (1);
             \draw[bend right=30] (4) to (3);
             \draw[bend right=30] (7) to (5);
             \draw[bend left=30] (7) to (3);
             \draw[bend left=30] (7) to (0);
             \draw[bend left=30] (2) to (8);
             \draw[bend left=30] (2) to (3);
             \draw[bend left=30] (6) to (5);
             \draw[bend left=30] (6) to (0);
             \draw[bend left=30] (8) to (0);
        \end{tikzpicture}
        \caption{The backedge graph of $T$ corresponding to the ordering $6 \prec 8 \prec 2 \prec 9 \prec 1 \prec 3 \prec 4 \prec 5 \prec 7$}
        \label{fig:not_second}
    \end{figure}
    
    One can check that in all $\diomega$-ordering of $T$, exactly one of $uv$ and $wx$ is forward. See the Python code in Figure~\ref{fig:not_code}.

    \begin{figure}[H]
        \centering
        \begin{python}
import itertools

T = [
  [0,1,1,1,1,1,0,0,0],
  [0,0,1,1,0,0,1,0,1],
  [0,0,0,1,1,1,1,0,0],
  [0,0,0,0,1,0,1,1,1],
  [0,1,0,0,0,1,1,1,0],
  [0,1,0,1,0,0,1,1,0],
  [1,0,0,0,0,0,0,1,1],
  [1,1,1,0,0,0,0,0,1],
  [1,0,1,0,1,1,0,0,0]
]

for P in itertools.permutations(range(9)):
  omega_at_most_two = True
  for i in range(9):
    for j in range(i+1, 9):
      for k in range(j + 1, 9):
        if T[P[j]][P[i]] and T[P[k]][P[j]] and T[P[k]][P[i]]:
          omega_at_most_two = False
  if omega_at_most_two:
    assert((P.index(6) < P.index(8)) != (P.index(7) < P.index(2)))
        \end{python}
        \caption{Python code to check that in every $\diomega$-ordering of $T$, exactly one of $uv$ or $wx$ is forward}
        \label{fig:not_code}
    \end{figure}

    By Lemma~\ref{lem:uplift}, there exists a tournament $T_3$ such that $T_{var} = \Delta(1, T, T_3)$ satisfies the desired property.
\end{proof}

We also need a gadget tournament to encode clauses, In order to do so, we build a tournament in which three fixed arcs are such that at least one of them is always backward. This will encode a disjunction.

\begin{lemma}\label{lem:T_clause}
    There exists a tournament $T_{clause}$ with $\diomega(T_{clause}) = 3$ and three disjoint arcs $uv, wx, yz$ such that:
    \begin{itemize}
        \item for every $\diomega$-ordering of $T_{clause}$, one of $uv, wx, yz$ is backward, and
        \item for every two arcs in $uv, wx, yz$, there exists a $\diomega$-ordering in which these two arcs are forward.
    \end{itemize}
\end{lemma}

Similarly to Lemma~\ref{lem:T_var}, this proof relies on computer help and we proceed by first finding such a tournament with clique number $2$, then leveraging it thanks to Lemma~\ref{lem:uplift}. 

\begin{proof}
    Let $T$ be the tournament with the following ordered graph as a backedge graph.

    \begin{figure}[H]
        \centering
        \begin{tikzpicture}
             \vertex (5) at (1,0) {$1$};
             \vertex (4) at (2,0) {$2$};
             \vertex (0) at (3,0) {$3$};
             \vertex (1) at (4,0) {$4$};
             \vertex (3) at (5,0) {$5$};
             \vertex (7) at (6,0) {$6$};
             \vertex (2) at (7,0) {$7$};
             \vertex (6) at (8,0) {$8$};
    
             \draw[bend right=30] (5) to (7);
             \draw[bend right=30] (5) to (2);
             \draw[bend right=30] (5) to (6);
             \draw[bend right=30] (4) to (3);
             \draw[bend left=30] (4) to (6);
             \draw[bend left=30] (0) to (7);
             \draw[bend left=30] (0) to (6);
             \draw[bend left=30] (1) to (6);
             \draw[bend left=30] (3) to (2);
        \end{tikzpicture}
        \caption{The backedge graph corresponding to the ordering $1 \prec 2 \prec 3 \prec 4 \prec 5 \prec 6 \prec 7 \prec 8$}
        \label{fig:or_first}
    \end{figure}
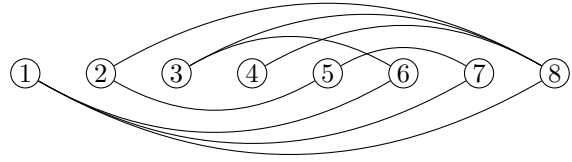
    
    Note that $T$ is not transitive and the ordering $1 \prec 2 \prec 3 \prec 4 \prec 5 \prec 6 \prec 7 \prec 8$ is such that $\omega(T^\prec) = 2$, thus $\diomega(T) = 2$. Let $uv = 5 \rightarrow 6$, $wx = 2 \rightarrow 4$ and $yz = 8 \rightarrow 3$.

    The ordering $1 \prec 2 \prec 3 \prec 4 \prec 5 \prec 6 \prec 7 \prec 8$ is a $\diomega$-ordering in which $uv$ and $wx$ are forward. See Figure~\ref{fig:or_first}.

    The ordering $4 \prec 7 \prec 5 \prec 8 \prec 1 \prec 2 \prec 6 \prec 3$ is a $\diomega$-ordering in which $uv$ and $yz$ are forward. See Figure~\ref{fig:or_second}.

    \begin{figure}[H]
        \centering
        \begin{tikzpicture}
             \vertex (5) at (5,0) {$1$};
             \vertex (4) at (6,0) {$2$};
             \vertex (0) at (8,0) {$3$};
             \vertex (1) at (1,0) {$4$};
             \vertex (3) at (3,0) {$5$};
             \vertex (7) at (7,0) {$6$};
             \vertex (2) at (2,0) {$7$};
             \vertex (6) at (4,0) {$8$};
    
             \draw[bend left=30] (1) to (6);
             \draw[bend left=30] (1) to (5);
             \draw[bend right=30] (4) to (1);
             \draw[bend right=30] (1) to (0);
             \draw[bend left=30] (2) to (4);
             \draw[bend left=30] (7) to (2);
             \draw[bend left=30] (0) to (2);
             \draw[bend right=30] (5) to (3);
             \draw[bend right=30] (3) to (0);
             \draw[bend left=30] (7) to (6);
             \draw[bend left=30] (7) to (5);
        \end{tikzpicture}
        \caption{The backedge graph of $T$ corresponding to the ordering $4 \prec 7 \prec 5 \prec 8 \prec 1 \prec 2 \prec 6 \prec 3$}
        \label{fig:or_second}
    \end{figure}
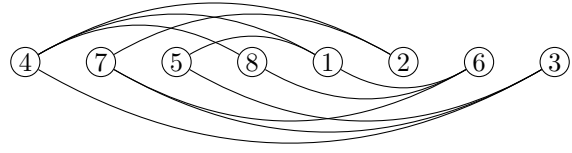
    
    The ordering $1 \prec 2 \prec 4 \prec 6 \prec 7 \prec 5 \prec 8 \prec 3$  is a $\diomega$-ordering in which $wx$ and $yz$ are forward. See Figure~\ref{fig:or_third}.

    \begin{figure}[H]
        \centering
        \begin{tikzpicture}
             \vertex (5) at (1,0) {$1$};
             \vertex (4) at (2,0) {$2$};
             \vertex (0) at (8,0) {$3$};
             \vertex (1) at (3,0) {$4$};
             \vertex (3) at (6,0) {$5$};
             \vertex (7) at (4,0) {$6$};
             \vertex (2) at (5,0) {$7$};
             \vertex (6) at (7,0) {$8$};
    
             \draw[bend right=30] (5) to (7);
             \draw[bend right=30] (5) to (2);
             \draw[bend right=30] (5) to (6);
             \draw[bend right=30] (4) to (3);
             \draw[bend left=30] (4) to (6);
             \draw[bend left=30] (7) to (1);
             \draw[bend left=30] (7) to (3);
             \draw[bend left=30] (1) to (6);
             \draw[bend left=30] (1) to (0);
             \draw[bend left=30] (2) to (0);
             \draw[bend left=30] (3) to (0);
        \end{tikzpicture}
        \caption{The backedge graph of $T$ corresponding to the ordering $1 \prec 2 \prec 4 \prec 6 \prec 7 \prec 5 \prec 8 \prec 3$}
        \label{fig:or_third}
    \end{figure}
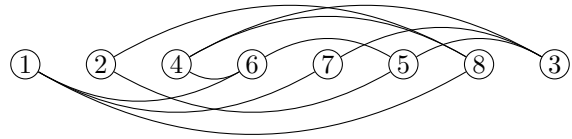

    One can check that in all $\diomega$-ordering of $T$, one of $uv$, $wx$, $yz$ is backward.
    See the Python code in Figure~\ref{fig:or_code}.

    \begin{figure}[H]
        \centering
        \begin{python}
import itertools

T = [
  [0,1,1,1,1,0,0,0],
  [0,0,1,1,0,1,1,0],
  [0,0,0,1,1,0,1,0],
  [0,0,0,0,1,1,1,0],
  [0,1,0,0,0,1,0,1],
  [1,0,1,0,0,0,1,1],
  [1,0,0,0,1,0,0,1],
  [1,1,1,1,0,0,0,0],
]

for P in itertools.permutations(range(8)):
  omega_at_most_two = True
  for i in range(8):
    for j in range(i+1, 8):
      for k in range(j+1, 8):
        if T[P[j]][P[i]] and T[P[k]][P[j]] and T[P[k]][P[i]]:
          omega_at_most_two = False
  if omega_at_most_two:
    assert(P.index(5) < P.index(4) or P.index(3) < P.index(1) or P.index(2) < P.index(7))
        \end{python}
        \caption{Python code to check that in every $\diomega$-ordering of $T$, one of $uv, wx, yz$ is backward}
        \label{fig:or_code}
    \end{figure}

    By Lemma~\ref{lem:uplift}, there exists a tournament $T_3$ such that $T_{clause} = \Delta(1, T, T_3)$ satisfies the desired property.
\end{proof}

We now make the reduction from 3-SAT.

\begin{theorem}
    The problem of deciding whether a given tournament $T$ verifies $\diomega(T) \le 3$ is NP-complete.
\end{theorem}

\begin{proof}
    Clearly, this problem is in NP. Let us prove it is NP-hard by reducing $3$-SAT to this problem.
    
    Let $\varphi$ be a $3$-SAT formula on variables $v_i$ for $i \in \intv{1}{n}$, and with clauses $C_j = C_{j,1} \lor C_{j,2} \lor C_{j,3}$ where $C_{j,k}$ are literals.
    Without loss of generality, we can suppose no variable appears more than once in the literals of a clause.
    For every variable $v_i$, we create a copy $A_i$ of $T_{var}$ (as defined in Lemma~\ref{lem:T_var}). Let $f^+_i$ and $f^-_i$ be the two arcs corresponding to $uv$ and $wx$ in $A_i$.
    For every clause $C_j$, we create a copy $B_j$ of $T_{clause}$ (as defined in Lemma~\ref{lem:T_clause}). Let $e^1_{j}, e^2_{j}, e^{3}_j$ be the three arcs corresponding to $uv, wx, yz$ in $B_j$.

    We consider the tournament $T = A_1 \Ra \cdots \Ra A_n \Ra T_3 \Ra B_1 \Ra \cdots \Ra B_m$, where $T_3$ is the tournament defined in Definition~\ref{def:T_k}, in which we revert some arcs as follows: for every literal $C_{j,k}$ corresponding to the variable $v_i$, let $ab = f^+_i$ if it is a positive literal and $ab = f^-_i$ otherwise, and let $cd = e^{k}_{j}$. We revert the arcs $ac, ad, bc$ and $bd$, so that after doing so, we have $ca, da, cb, db \in A(T)$.
    Let us prove that $\diomega(T) = 3$ if and only if $\varphi$ is satisfiable.
    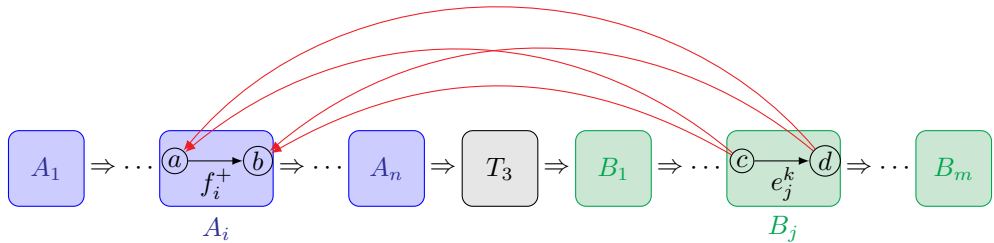
\begin{figure}[h]
        \centering
        \begin{tikzpicture}[
            bullet/.style={circle, fill, inner sep=1pt},
            rect/.style={rounded corners}
            ]
            \foreach \x in {1,3.5,5.5,7,8.5,11}
                \node at (\x+.25,.5) {$\Ra$} ;
            \foreach \x in {1,3.5,8.5,11}
                \node at (\x+.75,.5) {$\cdots$} ;

            \draw[rect, blue, fill=blue!20] (0,0) rectangle (1,1) ;
            \draw[rect, blue, fill=blue!20] (2,0) rectangle (3.5,1) ;
            \draw[rect, blue, fill=blue!20] (4.5,0) rectangle (5.5,1) ;
            \draw[rect, fill=gray!20] (6,0) rectangle (7,1) ;
            \draw[rect, Green, fill=Green!20] (7.5,0) rectangle (8.5,1) ;
            \draw[rect, Green, fill=Green!20] (9.5,0) rectangle (11,1) ;
            \draw[rect, Green, fill=Green!20] (12,0) rectangle (13,1) ;

            \node at (.5,.5) {\color{Blue}$A_1$} ;
            \node at (2.75,-.3) {\color{Blue}$A_i$} ;
            \node at (5,.5) {\color{Blue}$A_n$} ;
            \node at (6.5,.5) {$T_3$} ;
            \node at (8,.5) {\color{Green}$B_1$} ;
            \node at (10.25,-.3) {\color{Green}$B_j$} ;
            \node at (12.5,.5) {\color{Green}$B_m$} ;

            \vertex (a) at (2.2,.6) {$a$} ;
            \vertex (b) at (3.3,.6) {$b$} ;
            \vertex (c) at (9.7,.6) {$c$} ;
            \vertex (d) at (10.8,.6) {$d$} ;
            \node (e) at (10.25,.3) {$e_j^k$} ;
            \node (f) at (2.75,.3) {$f_i^+$} ;
            
            \draw[-latex] (a) to (b) ;
            \draw[-latex] (c) to (d) ;
            \draw[-latex, Red] (c) to[bend right=40] (a) ;
            \draw[-latex, Red] (c) to[bend right] (b) ;
            \draw[-latex, Red] (d) to[bend right=50] (a) ;
            \draw[-latex, Red] (d) to[bend right=40] (b) ;
        \end{tikzpicture}
        \caption{Construction of the tournament $T$ (when the literal $C_{j,k}$ is a positive instance of the variable $v_i$)}
    \end{figure}

    First, suppose $\diomega(T) = 3$ and let $\prec$ be a $\diomega$-ordering of $T$.
    Note that by Lemma~\ref{lem:complete}, for all $i \in \intv{1}{n}$ and $j \in \intv{1}{m}$, we have $V(A_i) \prec V(B_j)$.
    Let $\nu$ be the assignment such that $\nu(v_i)$ if and only if the arc $f^+_i$ is forward, or equivalently, if and only if the arc $f^-_i$ is backward.
    We will prove that $\nu$ satisfies $\varphi$.
    
    Let $C_j$ be a clause.
    By Lemma~\ref{lem:T_clause}, there exists $k \in \{1,2,3\}$ such that $e^k_{j} = cd$ is backward.
    Let $ab = f^+_i$ if $C^k_j$ is a positive literal and $ab = f^-_i$ otherwise.
    Then, $ab$ must be forward, for otherwise $\{a,b,c,d\}$ would induce a $K_4$ in $T^\prec$ since $ca, da, cb$ and $db$ are backward.
    Thus the literal $C^k_j$ is satisfied. Hence every clause is satisfied by $\nu$.

    \medskip

    Now, suppose $\varphi$ is satisfiable. Let $\nu$ be an assignment satisfying $\varphi$, and for each $j \in \intv{1}{m}$, let $k_j \in \{1,2,3\}$ be such that $C_{j,k_{j}}$ is satisfied. Consider an ordering $\prec$ of $T$ such that:
    \begin{itemize}
        \item $A_1 \prec \cdots \prec A_n \prec T_3 \prec B_1 \prec \cdots \prec B_m$,
        \item for each $i \in \intv{1}{n}$, the vertices of $A_i$ are ordered following a $\diomega$-ordering of $A_i$ in which $f^+_1$ is forward if and only if $\nu(v_1)$,
        \item the vertices of $T_3$ are ordered following a $\diomega$-ordering of $T_3$, and
        \item for each $j \in \intv{1}{m}$, the vertices of $B_j$ are ordered following a $\diomega$-ordering of $B_j$ in which, for $k \in \{1,2,3\}$, the arc $e^k_1$ is backward if and only if $k = k_j$.
    \end{itemize}

    Suppose there exists $K \subset V(T)$ such that $T^\prec[K] = K_4$. Note that $K$ cannot intersect $V(T_3)$, since $\omega(T^\prec_3) = 3$ and there are no edges between $V(T_3)$ and $V(T) \setminus V(T_3)$ in $T^\prec$.
    For $i,i' \in \intv{1}{n}$ distinct, $K$ cannot intersect both $V(A_i)$ and $V(A_{i'})$, as there are no edges between $A_i$ and $A_{i'}$ in $T^\prec$.
    For the same reason, $K$ cannot intersect both $V(B_j)$ and $V(B_{j'})$ when $j \neq j'$.
    Thus, $K \subset V(A_i) \cup V(B_j)$ for some integers $i \in \intv{1}{n}$ and $j \in \intv{1}{m}$.
    Since $\omega(A^\prec_i) = 3$, $K$ intersects $B_j$, and since $\omega(B^\prec_j) = 3$, $K$ intersects $A_i$.
    Vertices of $A_i$ have at most 2 in-neighbours in $B_j$ and $A_i \prec B_j$,  so $K$ intersects $B_j$ on at most $2$ vertices.
    Likewise, vertices of $B_j$ have at most $2$ out-neighbours in $A_i$ and $A_i \prec B_j$, so $K$ intersects $A_i$ on at most $2$ vertices.
    Thus $K = \{u,v,w,x\}$ with $v \prec u$ two distinct vertices of $A_i$ and $x \prec w$ two distinct vertices of $B_j$.
    This is only possible if there exists $k \in \{1,2,3\}$ such that $wx = e^k_j$ and $uv = f^+_i$ if $C^k_j$ is a positive literal and $uv = f^-_i$ otherwise.
    The fact that $e^k_j$ is backward implies $k = k_j$, but then $uv$ must be forward by construction, a contradiction. Therefore $\omega(T^\prec) < 4$.
\end{proof}

This proves Theorem~\ref{thm:NP} in the case $k=3$. Using Lemma~\ref{lem:uplift}, we can extend this to all larger of values $k$.

\section{Approximation Algorithms}\label{sec:approx}

Let $A_1$ and $D_1$ denote one-vertex tournaments, and for $n \ge 1$, let $D_{n+1} = \Delta(1,D_n,D_n)$, and denote $A_{n+1}$ the tournament made from $n$ copies $T_1, \dots, T_n$ of $A_n$ and $n+1$ additional vertices $v_0, \dots, v_n$ such that for $0 \le i < j \le n$, we have $v_i \La v_j$ and $v_i \Ra T_j$, and for $1 \le i < j \le n$, we have $T_i \Ra v_j$ and $T_i \Ra T_j$. These two families of tournaments were introduced by Kim $\cite{Kim}$ who showed that $\diomega(A_n) \ra \infty$ and $\diomega(D_n) \ra \infty$ as $n$ goes to infinity. Recently, Crew, Fan, Koerts, Moore, and Spirkl \cite{Spirkl} proved that for all $n\ge0$, every tournament with large enough directed clique number contains a subtournament isomorphic to $A_n$ or $D_n$.

\begin{theorem}[Crew, Fan, Koerts, Moore, \& Spirkl \cite{Spirkl}]\label{thrm:An_Dn}
    For every integer $n \ge 1$, there exists an integer $c_n \ge 0$ such that every tournament $T$ with $\diomega(T) \ge c_n$ contains a subtournament isomorphic to $A_n$ or $D_n$.
\end{theorem}

This means that a class of tournaments has bounded directed clique number if and only if it excludes members of both families as subtournaments. In particular, if a tournaments has very large directed clique number, then it contains a subtournament of bounded size with large directed clique number.

\begin{corollary}[Crew et al. \cite{Spirkl}]
    For every integer $k \ge 0$, there exists two integers $f(k) \ge 0$ and $\ell(k) \ge 0$ such that every tournament $T$ with $\diomega(T) \ge f(k)$ contains a subtournament $X$ verifying $\diomega(X) \ge k$ and $|X| \le \ell(k)$.
\end{corollary}

By the way, we note that such a statement cannot hold for digraphs in general.

\begin{proposition}
    For every two integers $k,\ell \ge 0$, there exists a digraph $D$ with $\diomega(D) \ge k$ and that contains no directed cycle of length at most $\ell$.
\end{proposition}

\begin{proof}
    Let $R_1$ denote a one-vertex tournament and, for $i \ge 1$, let $R_{k+1}$ be the digraph obtained from the directed cycle of length $\ell+1$ by substituting every vertex with a copy of $R_k$. We first show that for all $k \ge 1$, we have $\diomega(R_k) \ge k$.
    Clearly $\diomega(R_1)=1$. Let $k \ge 2$ be an integer and $\prec$ an ordering of $R_k$. The first vertex $v \in V(R_k)$ in that ordering has a copy $W$ of $R_{k-1}$ in its in-neighbourhood. As $v \prec W$ and $v \La W$, whenever $K$ is a clique in $W^\prec$, then $K \cup \{v\}$ is a clique in $R_k^\prec$. Hence $\diomega(R_k) \ge \diomega(R_{k-1})+1$, so by induction we have $\diomega(R_k) \ge k$.

    Suppose $R_k$ contains a directed cycle $C$ of length at most $\ell$ for some integers $k \ge 1$, and suppose $k$ is minimal with that property. Clearly $k>1$ thus $R_k$ is made from $\ell+1$ disjoint copies $S_0, \dots, S_\ell$ of $R_{k-1}$ by adding from $S_\ell$ to $S_0$ and from $S_{i-1}$ to $S_i$ for each $i \in [\ell]$. The cycle $C$ intersects at most $\ell < \ell+1$ of those copies, so it avoids one, say $S_0$ wlog. By minimality of $k$, the cycle $C$ is not included in one copy of $R_{i-1}$, so there exist $1 \le i<j \le \ell$ such that $C$ intersects both $S_i$ and $S_j$. However, there is no path from $S_j$ to $S_i$ in $R_k \setminus S_0$, a contradiction.
\end{proof}

We can also deduce approximation algorithms for the directed clique number of a tournament. Note that for $k \le 2$, it is easy to decide if a digraph $D$ verifies $\diomega(D) \ge k$ in polynomial-time.

\begin{corollary}\label{coro:approx}
    For every integer $k \ge 3$, there exists an integer $f(k) \ge 0$ and a polynomial-time algorithm that, given on input a tournament $T$, certifies either $\diomega(T) \ge k$ or $\diomega(T) < f(k)$.
\end{corollary}
\begin{proof}
    Let $k \ge 3$ be an integer. As $\diomega(A_n)$ and $\diomega(D_n)$ both go to infinity with $n$, there exists an integer $N(k) \ge 1$ such that $\diomega(D_{N(k)}) \ge k$ and $\diomega(A_{N(k)}) \ge k$. By \Cref{thrm:An_Dn}, there exists an integer $f(k) \ge 0$ such that every tournament $T$ with $\diomega(T) \ge f(k)$ contains a subtournament isomorphic to $A_{N(k)}$ or $D_{N(k)}$.
    
    Let $m(k) = \max \{ |V(A_{N(k)})|, |V(D_{N(k)})|\}$. Given a tournament $T$ one can check in time $O_k \big( |V(T)|^{m(k)} \big)$ if $T$ contains a subtournament isomorphic to $A_{N(k)}$ or $D_{N(k)}$ by iterating over all sets of $m(k)$ vertices. If $T$ does contain $A_{N(k)}$ or $D_{N(k)}$ then $\diomega(T) \ge k$, otherwise \Cref{thrm:An_Dn} certifies that $\diomega(T) < f(k)$.
\end{proof}

We next describe such an algorithm in the case $k=3$, with an explicit value for $f(k)$, and that, when the output is $\diomega(T) <  f(k)$, returns an ordering certifying it. We use the following algorithm as a subroutine.

\begin{proposition}[Klingelhoefer \& Newman \cite{KN24}]\label{algo:color}
    There exists a polynomial-time algorithm that, given on input a tournament $T$ with $\dic(T) \le 2$, computes a 10-dicolouring of $T$.
\end{proposition}

Klingelhoefer and Newman also prove that, in this statement, the 10 cannot be replaced by a 3 unless P=NP, and ask what is the best number of colours one can achieve. It is open whether there is such an algorithm for digraphs in general. We refer the reader to \cite{KN24} for existing results and further discussion on that matter.

\begin{proposition}
    There exists a polynomial-time algorithm that, given on input a tournament $T$ with $\diomega(T) \le 2$, computes an ordering $\prec$ of $T$ verifying $\omega(T^\prec) \le 100$.
\end{proposition}
\begin{proof}
    Let $T$ be a tournament with $\diomega(T) \le 2$. Fix a $\diomega$-ordering $\prec$ of $T$ (of course the algorithm has no access to $\prec$). Given two vertices $x, y \in V(T)$, we denote by $\Gamma_T(xy)$ the set $N_T^-(x) \cap N_T^+(y)$.
    \begin{claim}\label{claim:A-backward}
        If $x \prec y$ then $\dic(\Gamma_T(xy)) \le 2$.
    \end{claim}
    \begin{proofclaim}
        Let $\prec$ be a $\diomega$-ordering of $T$ and $x,y \in V(T)$ two vertices with $x \prec y$. The set $\Gamma_T(xy)$ can be partitioned into three parts $\Gamma_1, \Gamma_2, \Gamma_3$ such that $\Gamma_1 \prec x \prec \Gamma_2 \prec y \prec \Gamma_3$. As $x \La \Gamma_T(x,y) \La y$, in $T^\prec$ every vertex of $\Gamma_1 \cup \Gamma_2$ is adjacent to $y$, and every vertex of $\Gamma_3$ is adjacent to $x$. Since $T^\prec$ is triangle-free, the sets $\Gamma_1 \cup \Gamma_2$ and $\Gamma_3$ are stable in $T^\prec$, and thus acyclic in $T$. Hence $\dic(\Gamma_T(xy)) \le 2$. 
    \end{proofclaim}
    For every arc $xy \in A(T)$, we run the algorithm given by Proposition~\ref{algo:color} on the tournament $T[\Gamma_T(xy)]$. This either yields a $10$-dicolouring of $\Gamma_T(x,y)$, or fails which certifies that $\dic(\Gamma_T(x,y)) > 2$.
    If the second case occurs for some arcs $xy$, then we flip the arc $xy$ (change $x \ra y$ to $y \ra x$) and repeat the previous step for all arcs. This process terminates after $O(n^2)$ iterations because, by Claim~\ref{claim:A-backward} the arc $xy$ is backward in $\prec$, thus after flipping it the backedge graph is still triangle-free so Claim~\ref{claim:A-backward} still holds, hence the number of backward arc with respect to $\prec$ is strictly decreasing.
    We end up with a tournament $S$ obtained from $T$ by flipping a set of arcs $H$, such that $\omega(S^\prec)=2$ and for every arc $xy \in A(S)$, we have a 10-dicolouring of $\Gamma_S(xy)$. Note that the arcs in $H$ were all backward in $T^\prec$ so $H$ contains no triangle. We next show that $\dic(S) \le 50$.

    Suppose $x \in V(T)$ is the last vertex in $\prec$. Then $N_{S^\prec}(x) = N^+_S(x)$ and $N_{S^\prec}(x)$ is stable as $S^\prec$ is triangle-free, thus $N^+_S(x)$ is acyclic in $S$. Likewise, if $y \in V(T)$ is the last vertex in $\prec$ then $N^-_S(y)$ is acyclic in $S$. It is easy to find in polynomial time two vertices $x,y \in V(T)$ such that $N^-_S(x)$ and $N^+_S(y)$ are both acyclic in $T$.  Without loss of generality, assume $S$ is strongly connected, as otherwise we may handle each strong component separately. Compute a shortest directed path $x = z_0, z_1, \ldots, z_\ell = y$  from $x$ to $y$ in $S$, and define:
    \begin{align*}
        U_0 &= N^+_S[z_0], \\
        U_i &= \Gamma_S(z_{i-1} z_i) \cup \{z_i\} \quad \text{for } i \in [\ell], \\
        U_{\ell+1} &= N^-_S(z_\ell).
    \end{align*}

    For $v \in V(T) \setminus U_0$, we have $z_0 \ra v$ thus $i = \min \{ j \in [\ell] \colon z_{j-1} \ra v \}$ is well-defined, and then $z_{j-1} \ra x \ra z_j$ so $x \in U_i$. Hence $(U_i)_{0 \le i \le \ell+1}$ covers $V(S)$. Note that $U_0$ and $U_{\ell+1}$ are acyclic by choice of $x=z_0$ and $y=z_\ell$, and for $i\in [\ell]$, we have 10-dicolouring of $U_i$ as $z_i \Ra \Gamma_S(z_{i-1} z_i)$.

    \begin{claim}\label{claim:no-arc}
        For $5 \le i+5 \le j \le \ell+1$, there is no arc from $U_i$ to $U_j$ in $S$.
    \end{claim}
    \begin{proofclaim}
        Suppose $u \in U_i$ and $v \in U_j$ are such that $u \ra v$. Observe that either $u=z_i$ or $z_i \ra u$, and $v = z_j$ or $v \ra z_{j-1} \ra z_j$. In all cases, there is a path $P$ of length at most 4 from $z_i$ to $z_j$. Then $z_0, \dots, z_i, P, z_j, \dots, z_\ell$ is a path of length $i + 4 + (\ell-j) \le \ell-1$ from $z_0$ to $z_\ell$, which contradicts the minimality of $\ell$.
    \end{proofclaim}

    We now build a $50$-dicolouring of $S$. Partition the index set $\{0, \ldots, \ell+1\}$ into five residue classes $C_r = \{i \mid i \equiv r \pmod{5} \}$ for $r \in \{0,1,2,3,4\}$, and set $W_r = \bigcup_{i \in C_r} U_i$. For every two indices $i < j$ in $C_r$ we have $j \ge i+5$, so by \Cref{claim:no-arc} there is no arc from $U_i$ to $U_j$ in $S$, meaning all arcs between distinct blocks of $W_r$ go from $U_j$ to $U_i$. Ordering the blocks of $W_r$ in strictly decreasing order of index therefore makes every inter-block arc forward, so the backedge graph of $S[W_r]$ has no inter-block edges and $\dic(S[W_r]) \le \max_{i \in C_r} \dic(S[U_i]) \le 10$. Assigning a distinct set of $10$ colours to each $W_r$ yields a valid $50$-dicolouring of $S$, inducing an ordering $\prec$ of $V(T)$ with $\chi(S^\prec) \le 50$.

    It remains to show that $\omega(T^\prec) \le 100$. Let $K$ be a clique in $T^\prec$ and colour its vertices with the $50$-colouring of $S^\prec$. Consider two vertices $u \prec v$ in $K$ with the same colour. The pair $\{u,v\}$ is not an edge of $S^\prec$ (proper colouring), yet it is an edge of $T^\prec$, so the arc $v \to_T u$ is backward under $\prec$. If $vu \notin H$ then $S$ and $T$ agree on this arc, so $vu$ would also be a backward arc in $S$, giving an edge in $S^\prec$, a contradiction. Hence $vu \in A$, and by \Cref{claim:A-backward}, $vu$ is backward in every $\diomega$-ordering of $T$, i.e., $u \prec_\sigma v$ for all $\diomega$-orderings $\sigma$ of $T$.

    Suppose some colour class contains $m \ge 3$ vertices $v_1 \prec \cdots \prec v_m$ from $K$. For each pair $i < j$ the argument above gives $v_j v_i \in A$ and $v_i \prec_\sigma v_j$ in every $\diomega$-ordering $\sigma$. Fix any such $\sigma$. The arcs $v_j \to v_i$ for $i < j$ are all backward under $\sigma$, so the backedge graph $T^\sigma$ contains the complete graph $K_m$, giving $\omega(T^\sigma) \ge m \ge 3 > 2$. This contradicts $\omega(T^\sigma) = \diomega(T) \le 2$.
    Every colour class of the $50$-colouring therefore contributes at most $2$ vertices to $K$, so $|K| \le 2 \times 50 = 100$. Since every step of the algorithm runs in polynomial time, this completes the proof.
\end{proof}

We do not know if the 100 can be improved in this statement, or if we one can give small explicit bounds for $f(k)$ when $k > 3$ in Corollary \ref{coro:approx}. Additionally, it is open whether such algorithms exist for digraphs in general.

\section{The \Gya-Sumner Conjecture for Tournaments}\label{sec:GS}

This section is dedicated to the $\dic$-boundedness of tournament. In the first part, we establish new $\dic$-bounding tournaments, and use them to deduce a conjecture of Aboulker, Aubian, Charbit, and Lopes \cite{AACL23} which state that if a class of tournaments has bounded twin-width, then it is $\dic$-bounded. In the second part, we find non $\dic$-bounding tournaments that admits a backedge graph which is a forest, contradicting another conjecture posed in \cite{AACL23}, and use these tournaments to further disprove two conjectures of Kim \cite{Kim}.
 
\subsection{$\dic$-bounding tournaments}

The $\dic$-boundedness of tournament is inherently related to the $\chr$-boundedness of ordered graphs. In fact, for a tournament to be $\dic$-bounding, it suffices that any of its backedge graphs is $\chr$-bounding.

\begin{proposition}[Aboulker, Aubian, Charbit, \& Lopes \cite{AACL23}]\label{prop:backedge_bounding}
    Let $H$ be a tournament. If $H$ has an ordering $\prec$ such that $H^\prec$ is $\chr$-bounding, then $H$ is $\dic$-bounding.
\end{proposition}
\begin{proof}
    Suppose $H^\prec$ is $\chr$-bounding, that is, there exists a function $f$ such that $\chr(G) \le f(\ome(G))$ for all $H^\prec$-free ordered graphs $G$. Let $T$ be an $H$-free tournament and $<_T$ an $\diomega$-ordering of $T$. Then, the ordered graph $T^{<_T}$ is $H^\prec$-free, as otherwise $T$ is not $H$-free. Thus by hypothesis $\chr(T^{<_T}) \le f(\ome(T^{<_T}))$. Moreover $\dic(T) \le \chr(T^{<_T})$ by (\ref{eq:dic}), and $\omega(T^{<_T}) = \diomega(T)$ by the choice of $<_T$. Hence we have $\dic(T) \le f(\diomega(T))$, so $H$ is $\dic$-bounding.
\end{proof}

As any ordered graph is a backedge graph of a (unique) tournament, we can build a $\dic$-bounding tournament from every $\chr$-bounding ordered graph.
Several $\chr$-bounding ordered graphs are identified in \cite{patterns}. It is proven that every ordered star (tree with one vertex adjacent to every other vertex) is $\chr$-bounding. Further, given two $\chr$-bounding ordered graph $H_1$ and $H_2$, the concatenation of $H_1$ and $H_2$ is $\chr$-bounding, and also, nesting $H_1$ between two consecutive adjacent vertices in $H_2$ yields a $\chr$-bounding ordered graph. Additionally, the ordered graph with 4 vertices $a \prec b \prec c \prec d$ and the two edges $ac$ and $bd$ is shown to be $\chr$-bounding in \cite{cross}.

\medskip

We next prove that the tournament called $U_5$ in the literature is $\dic$-bounding. This is the tournament on 5 vertices, say $u_1, \dots, u_5$, such that for $1 \le i < j \le 5$ we have $u_i \la u_j$ if and only if $(i,j) = (1,4)$ or $(2,5)$. Given two tournaments $S$ and $T$ and a vertex $v \in V(S)$, the substitution of $v$ by $T$ in $S$ is the operation of replacing $v$ by a copy of $T$ in $S$ and adding arcs between $T$ and $S-v$ so that $N^-(v) \Ra T \Ra N^+(v)$. We say that a tournament is \emph{prime} if it cannot be obtained via a substitution from two smaller tournaments. For an integer $n \ge 1$, we denote by $T_n$ the tournament on $2n+1$ vertices, say $v_0, \dots, v_{2n}$, such that for each vertex $v_i$ we have $\{ v_{i-n}, \dots, v_{i-1} \} \Ra v_i \Ra \{ v_{i+1}, \dots, v_{i+n}\}$ where indices are taken modulo $2n+1$.

\begin{theorem}[Liu \cite{U5}]\label{thm:U5}
    Let $T$ be a prime tournament. Then, $T$ is $U_5$-free if and only if $T$ is isomorphic to $T_n$ for some $n \ge 1$, or $V(T)$ can be partitioned into three sets $X, Y, Z$ such that $X \cup Y$, $Y \cup Z$, and $Z \cup X$ are transitive in $T$.
\end{theorem}

It is easy to deduce that every prime $U_5$-free tournament is 2-dicolourarble. On the other hand, every non-prime $U_5$-free tournament can be obtained by substitution from two smaller $U_5$-free tournaments. More generally, the class of $U_5$-free tournaments is the closure under substitution of the class of prime $U_5$-free tournaments. It is proven in \cite{Kim} and \cite{AACL23} that if a class of tournament is $\dic$-bounded, then its closure under substitution is $\dic$-bounded too. Because a class of bounded dichromatic number is trivially $\dic$-bounded, it follows that:

\begin{corollary}
    The tournament $U_5$ is $\dic$-bounding.
\end{corollary}

We now present a new family of $\chr$-bounding ordered graphs. For every integer $n \ge 4$, we denote by $M_n$ the ordered graph with $n$ vertices, say $v_1, \dots, v_n$ in order, and two edges $v_1v_{n-1}$ and $v_2v_n$. Note that $M_4$ is the ordered graph that is proven to be $\chr$-bounding in \cite{cross}. We generalize their proof, which is based on the result hereunder, to show that all ordered graphs $M_n$ are $\chr$-bounding. Given two integers $\ell$ and $m$, we call $(\ell,m)$-banana any graph made of $m$ internally disjoint paths of length at least $\ell$ with the same endpoints.

\begin{theorem}[Scott \& Seymour \cite{Banana}]\label{thm:banana}
    For all integers $k,\ell,m \ge 0$, there exists an integer $f(k,\ell,m) \ge 0$ such that every graph $G$ with chromatic number $\chr(G) \ge f(k,\ell,m)$ contains a clique of size $k$ or an induced $(\ell,m)$-banana.
\end{theorem}

We also use the following classical result of Ramsey.

\begin{theorem}[Ramsey \cite{Ramsey}]\label{thm:Ramsey}
    For every positive integers $n$ and $c$, there exists an integer $R=R(n, c)$ such that for all $c$-edge-colorings of the $R$-clique $K_R$, there is a set $X \subseteq V(K_R)$ of $n$ vertices such that $K_R[X]$ is monochromatic.
\end{theorem}

We are now ready to prove our theorem.

\begin{theorem}\label{thrm:Mn}
    For every integer $n\ge4$, the ordered graph $M_n$ is $\chr$-bounding.
\end{theorem}
\begin{proof}
    Let $R = R(5n,3)$. Consider a $(4,R)$-banana $B$ with a vertex-ordering $\prec$. We show that $M_n$ is an induced ordered subgraph of $(B,\prec)$, then the desired result follows from \Cref{thm:banana}. The graph $B$ is made of $R$ paths of length at least 4 with endpoints $x$ and $y$. Denote $P_1, \dots, P_R$ the interior of these paths. Each $P_i$ has length at least 2, and in particular, contains a vertex that is adjacent to $x$ but not $y$, a vertex that is adjacent to $y$ but not $x$, and a vertex that is non-adjacent to both $x$ and $y$. We assign to every pair of paths $P_i,P_j$ one of three colours: 
    \begin{enumerate}
        \item[\color{Green}green] if $P_i$ is next to $P_j$, that is, $P_i \prec P_j$ or $P_j \prec P_i$;
        \item[\color{blue}blue] if $P_i$ surounds $P_j$, that is, there is bipartition $P_i^- \cup P_i^+$ of $P_i$ such that $P_i^- \prec P_j \prec P_i^+$, or if $P_j$ surounds $P_i$;
        \item[\color{red}red] if $P_i$ crosses $P_j$, that is, $\exists u_i,v_i \in P_i$ and $u_j,v_j \in P_j$ such that $u_i \prec u_j \prec v_i \prec v_j$ or $u_j \prec u_i \prec v_j \prec v_i$.
    \end{enumerate}
    This describes a 3-edge-colouring of a clique on $R$ vertices. By \Cref{thm:Ramsey}, there exists a colour $c \in \{ \text{{\color{Green}green}, {\color{blue}blue}, \color{red}red} \}$ ad a set of $5n$ paths $Q_1, \dots, Q_{5n}$ such that every pair among them is assigned the colour $c$.

    If $c=\text{\color{Green}green}$, then without loss of generality, we may assume that $Q_1 \prec \dots \prec Q_{5n}$. Each path $Q_i$, except at most one, is such that $x \prec Q_i$ or $Q_i \prec x$. The same folds for $y$. Hence $5n-2$ of these paths are either before $x$ and $y$, between $x$ and $y$, or after $x$ and $y$. By the pigeonhole principle, there are $\lceil \frac{5n-2}3 \rceil \ge n-2$ paths, say $S_1 \prec \dots \prec S_{n-2}$, in the same position relatively to $x$ and $y$. In each case, it is easy to build a set inducing a copy of $M_n$ by picking $x$, $y$, and in each path $S_i$, the vertex that is adjacent to $x$ but not $y$, or the vertex that is adjacent to $y$ but not $x$, or a vertex that is non-adjacent to both $x$ and $y$. See the figure below for examples.
    
    
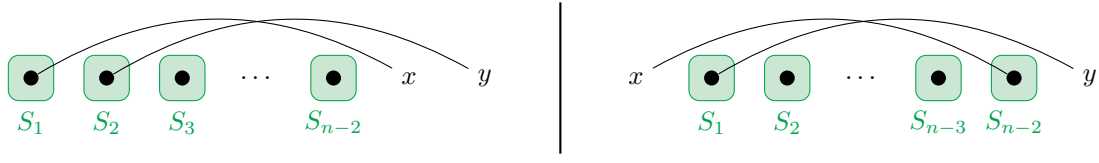
\begin{figure}[h]
        \centering
        \begin{tikzpicture}[
            bullet/.style={circle, fill, inner sep=2pt},
            rect/.style={rounded corners, Green, fill=Green!20}
            ]
            \draw[rect] (-.3,-.3) rectangle (0.3,.3) ;
            \draw[rect] (0.7,-.3) rectangle (1.3,.3) ;
            \draw[rect] (1.7,-.3) rectangle (2.3,.3) ;
            \draw[rect] (3.7,-.3) rectangle (4.3,.3) ;
            \node at (0,-.6) {\color{Green}$S_1$} ;
            \node at (1,-.6) {\color{Green}$S_2$} ;
            \node at (2,-.6) {\color{Green}$S_3$} ;
            \node at (4,-.6) {\color{Green}$S_{n-2}$} ;

            \node[bullet] (u1) at (0,0) {} ;
            \node[bullet] (u2) at (1,0) {} ;
            \node[bullet] (u3) at (2,0) {} ;
            \node at (3,0) {$\dots$} ;
            \node[bullet] (un) at (4,0) {} ;
            \node (x) at (5,0) {$x$} ;
            \node (y) at (6,0) {$y$} ;
            \draw (x) to[bend right] (u1) ;
            \draw (y) to[bend right] (u2) ;

            \draw[thick] (7,-1) -- (7,1) ;

            \draw[rect] ( 8.7,-.3) rectangle ( 9.3,.3) ;
            \draw[rect] ( 9.7,-.3) rectangle (10.3,.3) ;
            \draw[rect] (11.7,-.3) rectangle (12.3,.3) ;
            \draw[rect] (12.7,-.3) rectangle (13.3,.3) ;
            \node at ( 9,-.6) {\color{Green}$S_1$} ;
            \node at (10,-.6) {\color{Green}$S_2$} ;
            \node at (12,-.6) {\color{Green}$S_{n-3}$} ;
            \node at (13,-.6) {\color{Green}$S_{n-2}$} ;

            \node (x) at (8,0) {$x$} ;
            \node[bullet] (v1) at (9,0) {} ;
            \node[bullet] (v2) at (10,0) {} ;
            \node at (11,0) {$\dots$} ;
            \node[bullet] (v3) at (12,0) {} ;
            \node[bullet] (vn) at (13,0) {} ;
            \node (y) at (14,0) {$y$} ;
            \draw (x) to[bend left] (vn) ;
            \draw (y) to[bend right] (v1) ;
        \end{tikzpicture}
        \label{fig:green}
    \end{figure}

    If $c=\text{\color{blue}blue}$, then each $Q_i$ has a bipartition $Q_i^- \cup Q_i^+$ such that, without loss generality:
    \[
        Q_1^- \prec Q_2^- \prec \cdots \prec Q_{5n}^- \prec Q_{5n}^+ \prec \cdots \prec Q_2^+ \prec Q_1^+.
    \]
    Because $Q_i$ is connected, there is an edge $u_iv_i$ with $u_i \in Q_i^-$ and $v_i \in Q_i^+$. For each $i \in [5n]$, because $Q_i$ has length at least 2, either $x$ or $y$ is non-adjacent to $u_iv_i$. By pigeonhole principle, at least $\lceil \frac{5n}{2} \rceil \ge 2n$ of these edges are non-adjacent to the same end, say $x$ wlog. Denote these edges $u_1'v_1', \dots, u_{2n}'v_{2n}'$ so that we have
    \[
        u_1' \prec \cdots \prec u_{2n}' \prec v_{2n}' \prec \cdots \prec v_1'.
    \]
    We illustrate in the next figure how to build of set of vertices inducing of $M_n$ by distinguishing two cases, whether $x$ is between $u_n'$ and $v_n'$ (left in the figure), or not (right in the figure).
        
    
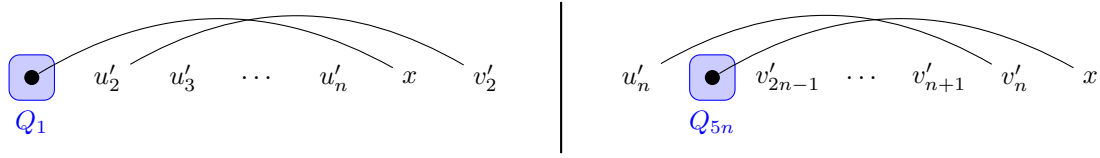
\begin{figure}[h]
        \centering
        \begin{tikzpicture}[
            bullet/.style={circle, fill, inner sep=2pt},
            rect/.style={rounded corners, blue, fill=blue!20}
            ]
            \draw[rect] (-.3,-.3) rectangle (.3,.3) ;
            \node (Q1) at (0,-.6) {\color{blue}$Q_1$} ;
            \node[bullet] (z) at (0,0) {} ;
            \node (u2) at (1,0) {$u_2'$} ;
            \node (u3) at (2,0) {$u_3'$} ;
            \node (u) at (3,0) {$\dots$} ;
            \node (un) at (4,0) {$u_n'$} ;
            \node (x) at (5,0) {$x$} ;
            \node (v2) at (6,0) {$v_2'$} ;
            \draw (z) to[bend left] (x) ;
            \draw (u2) to[bend left] (v2) ;

            \draw[thick] (7,-1) -- (7,1) ;

            \draw[rect] (8.7,-.3) rectangle (9.3,.3) ;
            \node (Q4) at (9,-.6) {\color{blue}$Q_{5n}$} ;
            \node (un) at (8,0) {$u_n'$} ;
            \node[bullet] (z) at (9,0) {} ;
            \node (v2) at (10,0) {$v_{2n-1}'$} ;
            \node (v) at (11,0) {$\dots$} ;
            \node (v1) at (12,0) {$v_{n+1}'$} ;
            \node (vn) at (13,0) {$v_n'$} ;
            \node (x) at (14,0) {$x$} ;
            \draw (z) to[bend left] (x) ;
            \draw (un) to[bend left] (vn) ;
        \end{tikzpicture}
        \label{fig:blue}
    \end{figure}

    If $c = \text{\color{red}red}$, that is, the paths $Q_1, \dots, Q_{5n}$ pairwise cross. For each $i \in [5n]$, let $a_i$ be the first vertex of $Q_i$ in the ordering, and without loss of generality, assume that $a_1 \prec \cdots \prec a_{5n}$. The path $Q_{5n}$ crosses all other paths, in particular, every path has a vertex after $a_{5n}$. For $i \in [5n-1]$, the path $Q_i$ is connected and have vertices before and after $a_{5n}$, hence there is an edge $u_iv_i \in Q_i$ such that $u_i \prec a_{5n} \prec v_i$. We renumber all paths except $Q_{5n}$ so that now $u_1 \prec \cdots \prec u_{5n-1}$. For $1 \le i < j \le 4n$, if the edges $u_i$ and $v_j$ cross ({\it i.e.} $u_i \prec u_j \prec v_i \prec v_j$), then the vertices $u_i, u_j, u_{4n+1}, \dots, u_{5n-1}, a_{5n}, v_i, v_j$ induce a copy of $M_n$. Therefore, it holds that $v_{4n} \prec \cdots \prec v_1$.
    
      Let us now show that the edge $u_{2n}v_{2n}$ is crossed by a at least $2n-1$ paths. If $Q_{2n+1}, \dots, Q_{4n}$ all cross $u_{2n}v_{2n}$ we are done, so suppose there is an index $i \in [2n+1,4n]$ such that $Q_i$ does not cross $u_{2n}v_{2n}$. Because $u_i \in Q_i$ is between $u_{2n}$ and $v_{2n}$, the entirety of $Q_i$ lies between $u_{2n}$ and $v_{2n}$. However, the paths $Q_1, \dots Q_{2n-1}$ cross $Q_i$, thus they have a vertex between $u_{2n}$ and $v_{2n}$, so they all cross the edge $u_{2n}v_{2n}$. Let $x_1 y_1,\cdots, x_{2n-1} y_{2n-1}$ be edges from different paths (in particular they are disjoint and non adjacent) that cross $u_{2n}v_{2n}$. We name the vertices so that $u_{2n} \prec x_1 \prec \cdots \prec x_{2n-1} \prec v_{2n}$. If $y_n \prec u_{2n}$, then $y_n, u_{2n}, x_4, \cdots, x_n, v_{2n}$ induce a copy of $M_n$, otherwise $u_{2n}, x_n,\cdots, x_{2n-4}, v_{2n}, y_n$ induce a copy of $M_n$.
\end{proof}

Further, Bria\'nski, Davies, and Walczak \cite{matchings} announced a proof that all ordered matchings (graphs with max degree 1) are $\chr$-bounding. This result generalizes \Cref{thrm:Mn} and implies that every tournament with a backedge graph that is a matching is $\dic$-bounding. We next deduce another conjecture of Aboulker, Aubian, Charbit, and Lopes \cite{AACL23}. Two other proofs of the same statement were independently found by Tang and Zhang using LLM \cite{chaoliang}. We refer to \cite{twinwidth} for the definition of the twin-width of a tournament.

\begin{theorem}\label{thm:twinwidth}
    Let $\cal T$ be a class of tournaments. If $\cal T$ has bounded twin-width, then $\cal T$ is $\dic$-bounded.
\end{theorem}

\begin{proof}
    Geniet and Thomass\'e \cite{twinwidth} proved that there are three (explicit) classes of tournaments $\cal F_=$, $\cal F_\le$, and $\cal F_\ge$ such that a class of tournament $\cal T$ has bounded twin-width if and only if it there are members of each classes $\cal F_=$, $\cal F_\le$, and $\cal F_\ge$ that are subtournaments of no tournaments in $\cal T$. In particular, if a class of tournament $T$ has bounded twin-width, then it is included in $\Forb(H)$ for some tournament $H \in \cal F_=$. Moreover, the tournaments in the class $\cal F_=$ all have a backedge graph that is a matching \cite{twinwidth}, so by \Cref{prop:backedge_bounding} and the aforementioned result of Davies, Bria\'nski, and Walczak, the tournament $H$ is $\dic$-bounding. Therefore, $\Forb(H)$ and $\cal T$ are $\dic$-bounded.
\end{proof}

It is open whether this is true for classes of digraphs in general. Moreover, in the undirected case, a graph class $\cal G$ with bounded twin-width is polynomially $\chr$-bounded \cite{bourneuf} (that is, there is a polynomial function $f$ such that $\chr(G) \le f(\omega(G))$ for all $G \in \cal G$). We may wonder if this holds for classes of tournaments as well.

\subsection{Non $\dic$-bounding tournaments}

The goal of this section is to discover non $\dic$-bounding tournaments. For this purpose, we need tournaments with bounded directed clique number and large dichromatic number. The next inequality help us achieve that.

\begin{proposition}[Nguyen, Scott, \& Seymour \cite{NSS}]\label{prop:ineq}
    Let $T$ be a tournament. For every ordering $\prec$ of $T$, we have $\chr(T^\prec) \le \dic(T) \cdot \ome(T^\prec)$.
\end{proposition}

We construct tournaments with bounded directed clique number and large dichromatic number as follows. Let $G$ be an ordered graph with clique number $\ome(G)=2$ and chromatic number $\chr(G)=k$, and let $T$ be the tournament that admits $G$ has a backedge graph. Then, we have $\diomega(T)=2$ and $\dic(T) \ge \frac k 2$ by \Cref{prop:ineq}. Similarly to the graph case, every $\dic$-bounding tournament must be a forest in the following sense:

\begin{proposition}[Aboulker, Aubian, Charbit, \& Lopes \cite{AACL23}]
    Let $H$ be a tournament. If $H$ is $\dic$-bounding, then there exists an ordering $\prec$ of $H$ such that $H^\prec$ is a forest.
\end{proposition}
\begin{proof}
    Suppose $H$ is $\dic$-bounding, and let $f$ be a function such that $\dic(T) \le f(\diomega(T))$ for all $H$-free tournaments $T$. Consider an ordered graph $G$ with clique number $\ome(G)=2$, chromatic number $\chr(G) > 2 f(2)$, and no cycle of length at most $|V(H)|$; such graphs exist by \cite{girth}. Let $T$ be the tournament that admits $G$ has a backedge graph. We have $\diomega(T)=2$ and $\dic(T)>f(2)$ by \Cref{prop:ineq}, hence $T$ is not $H$-free. Therefore some backedge graph $H^\prec$ of $H$ is an induced subgraph of $G$, so $H^\prec$ has no cycle of length at most $|V(H)|$, and thus it is a forest.
\end{proof}

Aboulker, Aubian, Charbit, and Lopes conjectured that the converse holds, that is, every tournament that has backedge graph which is a forest is $\dic$-bounding. In the rest of this section, we describe two families of tournaments that are not $\dic$-bounded, and use them to infer more non $\dic$-bounding tournaments.

\medskip

Let $(G,\prec)$ be an ordered graph and $\phi$ an embedding of $G$ in $\mathbb Z$, that is, a function $\phi \colon V(G) \to \mathbb{Z}$ which is increasing with respect to $\prec$. For each edge $uv \in E(G)$, we denote $\phi(uv) = |\phi(u) - \phi(v)|$. Given two positive integers $r$ and $s$, we say that two edges $e,f \in  E(G)$ are \emph{$(r,s)$-comparable} if $\frac 1 r \phi(f) \le \phi(e) \le r \cdot \phi(f)$ and there is a path of length at most $s$ connecting $e$ and $f$. We say that an ordered graph $G$ is \emph{$(r,s)$-incomparable} if it admits an embedding such that no pair of (distinct) edges is $(r,s)$-comparable. Note that if $G$ is $(r,s)$-incomparable, then $G$ is $(r',s')$-incomparable for all $r' \le r$ and $s' \le s$. This notion can be interpreted as a refinement of the girth (the minimum size of a cycle in a graph) for ordered graphs, as evidenced by the next lemma.

\begin{lemma}
    Let $r \ge 3$ be an integer and $G$ an $(r-1, \lfloor \frac r 2 \rfloor -1)$-incomparable ordered graph. Then, $G$ has no cycle of length at most $r$.
\end{lemma}

\begin{proof}
    Suppose that $G$ contains a cycle $C$ of length at most $r$, and let $\phi$ be an embedding of $G$ in $\mathbb Z$. Denote respectively $x$ and $y$ the first and the last vertices of $C$ in the ordering, and let $M = \phi(y) - \phi(x)$. For every edge $uv \in C$, we have $\phi(uv) = |\phi(v)-\phi(u)| \le \max \phi(C) - \min \phi(C) = M$. The vertices $x$ and $y$ are connected in $C$ by two disjoint paths $P = u_0 \cdots u_k$ and $P'$ of length at most $r-1$. Let $e$ be an edge of $P$ maximizing $\phi$. Then
    \[
        \varphi(e) \le M = \phi(y) - \phi(x) = \sum_{i=1}^k \phi(u_i) - \phi(u_{i-1}) \le \sum_{i=1}^k \phi(u_{i-1}u_i) \le k \cdot \phi(e) \le (r-1) \phi(e).
    \]    
    Likewise for $f$ an edge of $P'$ maximizing $\phi$, we have $\varphi(f) \le M \le (r-1) \phi(f)$. It follows that $\frac{1}{r-1} \phi(f) \le \phi(e) \le (r-1) \phi(f)$. Moreover, the edges $e$ and $f$ are connected by two disjoint paths consisting of the remaining $r-2$ edges of $C$. One of these two paths has length at most $\lfloor \frac{r-2}{2} \rfloor$, hence $e$ and $f$ are $(r-1, \lfloor \frac r 2 \rfloor -1)$-comparable.
\end{proof}

In particular, the triangle is not $(2,0)$-incomparable, so for all $r\ge2$ and $s\ge0$, every $(r,s)$-incomparable ordered graph has clique number at most 2. As for the class of graphs of large girth, the class of all $(r,s)$-incomparable ordered graphs has unbounded chromatic number, and thus is not $\chr$-bounded.

\begin{theorem}[Berger et al. \cite{heroes}]
    For every positive integers $r$ and $s$, the class of all $(r,s)$-incomparable ordered graphs is not $\chr$-bounded.
\end{theorem}

It follows that, for every $r,s>0$, the class of all tournaments that have an $(r,s)$-incomparable backedge graph is not $\dic$-bounded. Therefore, if $H$ is a $\dic$-bounding tournament, then $H$ must have an $(r,s)$-incomparable backedge graph for all $r,s>0$. Furthermore, as $H$ has finitely many backedge graphs, it has a backedge graph which is $(r,r)$-incomparable for infinitely many values of $r>0$, and thus, which is $(r,s)$-incomparable for all $r,s>0$. We say that an ordered graph is \emph{incomparable} if it is $(r,s)$-incomparable for all $r,s>0$. The following proposition gives a characterization of incomparable graphs that it is easy to check.

\begin{proposition}[Kim \& Kim \cite{KimKim}]
    An ordered graph $(G,\prec)$ is incomparable if and only if there is a bipartition $A,B$ of $V(G)$ such that $A \prec B$, the edges between $A$ and $B$ are in different connected components of $G$, and $(G[A],\prec)$ and $(G[B],\prec)$ are incomparable.
\end{proposition}

We can now verify if a tournament has an incomparable backedge graph by iterating over all possible orderings and checking if this property holds. We were able to identify six tournaments that have a backedge graph which is a forest, but no incomparable backedge graph. Therefore, these are non $\dic$-bounding "forest" tournaments which contradicts the aforementionned conjecture of Aboulker, Aubian, Charbit, and Lopes \cite{AACL23}. We represent these tournaments in \Cref{fig:incomparable} by drawing, for each, a backedge graph which is a forest.

\begin{figure}[h]
    \centering
    \begin{tikzpicture}[bullet/.style={circle, fill, inner sep=2pt},nodes=bullet]
    
        \foreach \x in {1,...,5} \node at (\x,4) {} ;
        \draw (2,4)  to[bend left] (5,4)  to[bend left] (1,4)  to[bend left] (4,4)  ;

        \foreach \x in {0.5,...,5.5} \node at (\x,2) {} ;
        \draw (1.5,2)  to[bend right] (5.5,2)  to[bend left] (2.5,2)  ;
        \draw (3.5,2)  to[bend right] (0.5,2)  to[bend left] (4.5,2)  ;

        \foreach \x in {0.5,...,5.5} \node at (\x,0) {} ;
        \draw (1.5,0)  to[bend left] (4.5,0)  ;
        \draw (2.5,0)  to[bend right] (5.5,0)  to[bend right] (0.5,0)  to[bend right] (3.5,0)  ;

        \foreach \x in {8,...,14} \node at (\x,4) {} ;
        \draw (9,4)  to[bend right] (14,4)  to[bend left] (10,4) to[bend left]  (12,4)  to[bend right] (8,4)  to[bend left] (13,4)  ;

        \foreach \x in {8,...,14} \node at (\x,2) {} ;
        \draw (9,2)  to[bend left] (13,2)  ;
        \draw (11,2)  to[bend right] (14,2) to[bend right] (8,2)  to[bend right] (10,2) to[bend left] (12,2) ;

        \foreach \x in {8,...,14} \node at (\x,0) {} ;
        \draw (9,0) to[bend left] (13,0) to[bend right] (10,0) to[bend left] (8,0) to[bend left] (14,0) to [bend left] (11,0) ;
        
        \end{tikzpicture}
    \caption{(backedge graphs of) new non $\dic$-bounding tournaments}
    \label{fig:incomparable}
\end{figure}
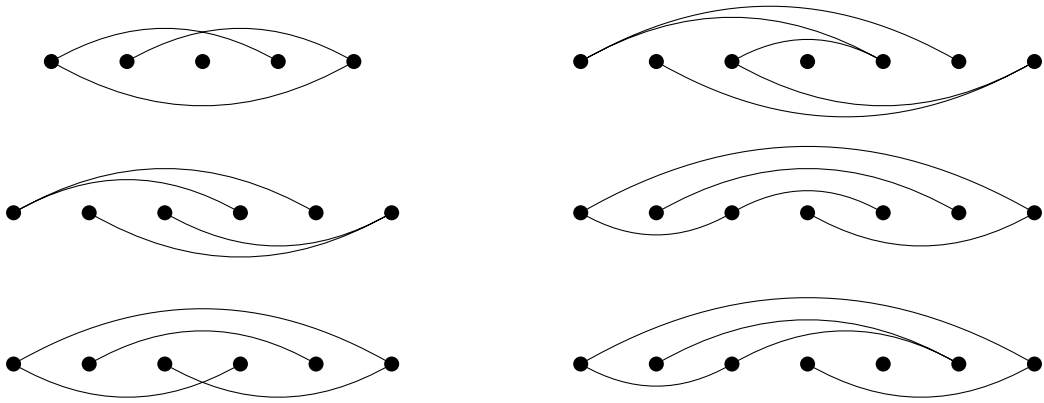

There are many ways to build graphs with bounded clique number and large chromatic number. One of the first constructions was published by Tutte under the pseudonym Blanche Descartes \cite{BD}. It is defined recursively as follows. Let $B_1$ be the one-vertex graph, and for $k \ge 1$, let $B_{k+1}$ be the graph made of $\binom{k \cdot |V(B_k)|}{|V(B_k)|}$ disjoint copies of $B_k$ and $k |V(B_k)|$ additional vertices, connecting each copy of $B_k$ to a distinct set of $|V(B_k)|$ additional vertices via a perfect matching. A simple induction show that for all $k \ge 1$, $B_k$ is triangle-free and $\chr(B_k)=k$ .

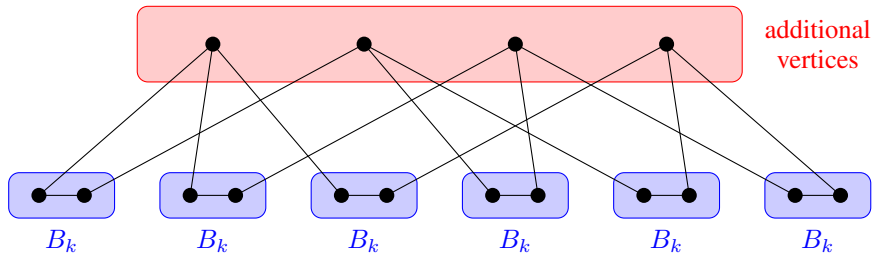
\begin{figure}[h]
    \centering
    \begin{tikzpicture}[
        bullet/.style={circle, fill, inner sep=2pt},
        blue_rect/.style={rounded corners, blue, fill=blue!20},
        red_rect/.style={rounded corners, red, fill=red!20},
        ]
    \draw[red_rect] (1,1.5) rectangle (9,2.5) ;
    \foreach \x in {2,4,...,8} \node[bullet] at (\x,2) {} ;
    \node at (10,2.2) {\color{red}additional} ;
    \node at (10,1.8) {\color{red}vertices} ;

    \foreach \x in {0,2,...,10} {
        \draw[blue_rect] (\x-.7,-.3) rectangle (\x+.7,.3) ;
        \node[bullet] at (\x-.3,0) {} ;
        \node[bullet] at (\x+.3,0) {} ;
        \draw (\x-.3,0) to (\x+.3,0) ;
        \node at (\x,-.6)  {\color{blue}$B_k$} ;
    }

    \draw (-.3,0) to (2,2) ;
    \draw (+.3,0) to (4,2) ;
    \draw (1.7,0) to (2,2) ;
    \draw (2.3,0) to (6,2) ;
    \draw (3.7,0) to (2,2) ;
    \draw (4.3,0) to (8,2) ;
    \draw (5.7,0) to (4,2) ;
    \draw (6.3,0) to (6,2) ;
    \draw (7.7,0) to (4,2) ;
    \draw (8.3,0) to (8,2) ;
    \draw (9.7,0) to (6,2) ;
    \draw (10.3,0) to (8,2) ;

    \end{tikzpicture}
    \caption{Blanche Descartes' construction ($k=2$)}
    \label{fig:BD}
\end{figure}

We recursively define an ordering $\prec_k$ of $B_k$ for every $k\ge1$. Let $\prec_1$ be the unique ordering of $B_1$, and for each $k \ge 1$, we define $\prec_{k+1}$ by first taking the additional vertices in any order, and then each copy of $B_k$, successively, in the order $\prec_k$. For $k\ge1$, denote $TB_k$ the tournament that admits $(B_k, \prec_k)$ as a backedge graph. By \Cref{prop:ineq}, we have $\diomega(TB_k)=2$ and $\dic(TB_k) \ge \frac k 2$. Therefore, every $\dic$-bounding tournament must be a subtournament of $TB_k$ for $k$ large enough, and all subtournaments of some $TB_k$ satisfies the next property:

\begin{proposition}
    Let $H$ be a $\dic$-bounding tournament. There exist a partition $A_0,\dots,A_k$ of $V(H)$ such that:
    \begin{enumerate}[(i)]
        \item $A_0$ is transitive ;
        \item $A_i \Ra A_j$ for $1 \le i<j \le k$ ;
        \item every vertex not in $A_0$ has at most one arc to $A_0$ ; and
        \item for each $i \in [k]$, every vertex in $A_0$ has at most one arc from $A_i$.
    \end{enumerate}
\end{proposition}

\begin{proof}
    We first deal with the case that $H$ is strongly connected. Let $f$ be a function such that $\dic(T) < f(\diomega(T))$ for all $H$-free tournaments $T$. The tournament $T_{2f(2)}$ defined above verifies $\diomega(T_k)=2$ and $\dic(T_k) \ge f(2)$, thus $T_{2f(2)}$ is not $H$-free. Consider the least integer $m$ such that $T_m$ contains a copy of $H$. If $m=1$, then $H$ has only one vertex, and the statement is trivial, so suppose $m>1$. Recall that $(B_m,\prec_m)$ is a backedge graph of $T_m$. Let $A_1 \prec_m \dots \prec_m A_k$ be the vertex-sets inducing disjoint copies of $B_{m-1}$ in $B_m$, and let $A_0$ be the additional vertices in $B_m$. It is easy to check that this partition satisfies properties $(i)$ to $(iv)$.
    
    For $0 \le i \le k$, denote $A_i' = A_i \cap V(H)$. Note that $A_0', \dots , A_k'$ still verify properties $(i)$ to $(iv)$, however some sets may be empty. We would like to remove all empty sets from this near-partition, but if we remove the set $A_0'$, property (i) may break. Let $i$ be the least integer such that $A_i' \ne \emptyset$, and toward a contradiction, suppose that $i>0$. By property (ii), we have $A_i' \Ra (H \setminus A_i')$. The strong connectivity of $H$ implies that $H \setminus A_i' = \emptyset$. Hence $H$ is contained in the set $A_i$ that induces a copy of $B_{m-1}$. This contradicts the minimality of $m$.

    If $H$ is not strongly connected, denote by $H'$ the strongly connected component of $H$ such that $H' \Ra (H \setminus H')$. As $\Forb(H') \subset \Forb(H)$, the tournament $H'$ is $\dic$-bounding too. Let $A_0, \dots, A_{k-1}$ be a  partition of $H'$ that satisfies the proposition and take $A_k = H \setminus H'$. It is easy to see that this partition verifies properties $(i)$ to $(iv)$.
\end{proof}

We were able to find tournaments that do not admit such partition, and thus are not $\dic$-bounding. They are represented in \Cref{fig:non}. There are many others constructions of triangle-free graphs with large chromatic number, for which we can try a similar approach, but we were not able to find any other non $\dic$-bounding tournaments from them. In \Cref{fig:todo}, we represent the smallest tournaments for which this is open.

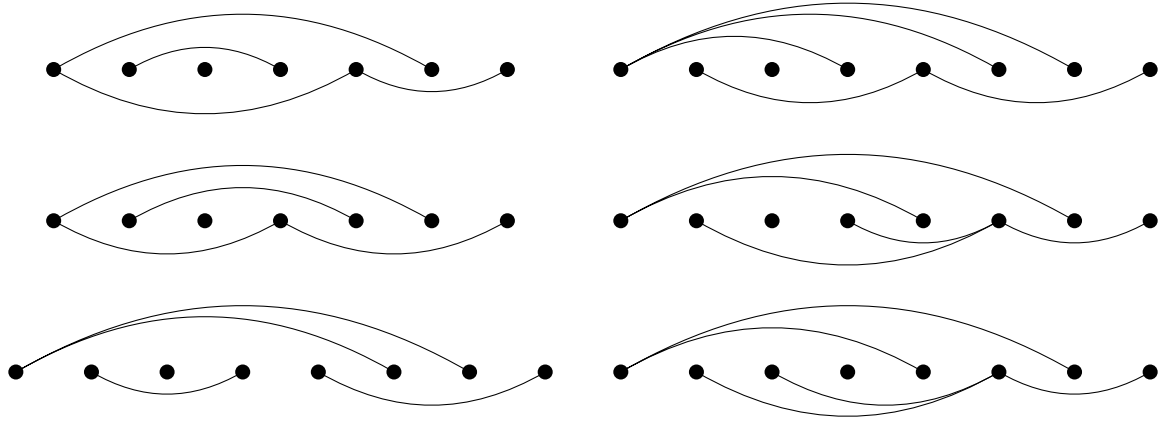
\begin{figure}[h]
    \centering
    \begin{tikzpicture}[bullet/.style={circle, fill, inner sep=2pt},nodes=bullet]

        \foreach \x in {0.5,...,6.5} \node at (\x,4) {} ;
        \draw (6.5,4)  to[bend left] (4.5,4) to[bend left] (0.5,4) to[bend left] (5.5,4) ;
        \draw (1.5,4) to[bend left] (3.5,4) ;
        
        \foreach \x in {0.5,...,6.5} \node at (\x,2) {} ;
        \draw (6.5,2)  to[bend left] (3.5,2)  to[bend left] (0.5,2)  to[bend left] (5.5,2)  ;
        \draw (1.5,2)  to[bend left] (4.5,2)  ;

        \foreach \x in {0,...,7} \node at (\x,0) {} ;
        \draw (1,0) to[bend right] (3,0) ;
        \draw (4,0) to[bend right] (7,0) ;
        \draw (5,0) to[bend right] (0,0) to[bend left] (6,0) ;
        
        \foreach \x in {8,...,15} \node at (\x,4) {} ;
        \draw (8,4) to[bend left] (11,4) ;
        \draw (9,4) to[bend right] (12,4) to[bend right] (15,4) ;
        \draw (13,4) to[bend right] (8,4) to[bend left] (14,4) ;

        \foreach \x in {8,...,15} \node at (\x,2) {} ;
        \draw (14,2) to[bend right] (8,2) to[bend left] (12,2) ;
        \draw (15,2) to[bend left] (13,2) to[bend left] (11,2) ;
        \draw (13,2) to[bend left] (9,2) ;

        \foreach \x in {8,...,15} \node at (\x,0) {} ;
        \draw (14,0) to[bend right] (8,0) to[bend left] (12,0) ;
        \draw (15,0) to[bend left] (13,0) to[bend left] (10,0) ;
        \draw (13,0) to[bend left] (9,0) ;
        
    \end{tikzpicture}
    \caption{more (backedge graphs of) new non $\dic$-bounding tournaments}
    \label{fig:non}
\end{figure}

We now refute Conjecture 5.6.1 (and thus Conjecture 5.5.10, which is stronger) from \cite{Kim}. This conjecture states that if $T_1$ and $T_2$ are subtournaments of respectively $D_n$ and $A_n$ for some $n \ge 0$ such that $T_2$ has an incomparable backedge graph and a backedge graph that is a "shift graph" (see \cite{Kim} for the definition), then $\Forb(T_1,T_2)$ has bounded dichromatic number. Consider $T_1 = D_3$ and $T_2$ the tournament represented top left of Figure \ref{fig:non}. It is easily checked that $T_2$ verifies the hypothesis of the conjecture, and that the tournaments $TB_k$ contain neither $T_1$ nor $T_2$. Hence $\Forb(T_1,T_2)$ has unbounded dichromatic number.

\begin{figure}[h]
    \centering
    \begin{tikzpicture}[bullet/.style={circle, fill, inner sep=2pt},nodes=bullet]
    
        \foreach \x in {0,...,5} \node at (\x,0) {} ;
        \draw (0,0) to[bend left] (3,0) to[bend left] (5,0) ;
        \draw (1,0) to[bend right] (4,0) ;

        \foreach \x in {0,...,5} \node at (\x+8,0) {} ;
        \draw (13,0) to[bend right] (8,0) to[bend left] (11,0) ;
        \draw (9,0) to [bend right] (12,0) ;

        \foreach \x in {0,...,5} \node at (\x+4,2) {} ;
        \draw (8,2) to[bend right] (4,2) to[bend left] (7,2) ;
        \draw (5,2) to[bend right] (9,2) ;

        \foreach \x in {0,...,5} \node at (\x,4) {} ;
        \draw (4,4) to[bend right] (0,4) to[bend left] (3,4) ;
        \draw (2,4) to[bend right] (5,4) ;

        \foreach \x in {0,...,5} \node at (\x+8,4) {} ;
        \draw (13,4) to[bend right] (9,4) to[bend left] (12,4) ;
        \draw (8,4) to[bend right] (11,4) ;
        
    \end{tikzpicture}
    \caption{minimal tournaments for which it is open whether they are $\dic$-bounding or not.}
    \label{fig:todo}
\end{figure}
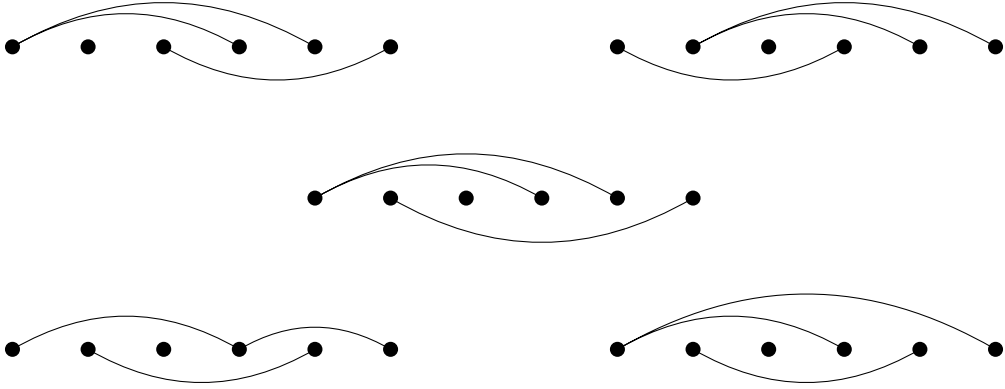

Let us end this section by proposing a possible (though hazardous) characterisation of $\dic$-bounding tournaments. Note that the backward implication holds as it is exactly  Proposition \ref{prop:backedge_bounding}.

\begin{conjecture}\label{conj:GS2}
    A tournament is $\dic$-bounding if and only if it has a backedge graph which is $\chr$-bounding.
\end{conjecture}

\section{$\diomega$-critical Tournaments}\label{sec:crit}

The results in this section were found with Aboulker and Charbit, who chosed not to appear as authors.

Let $T$ be a tournament and $k\ge 1$ an integer. We say that $T$ is \emph{$k$-$\dic$-critical} if $\dic(T) = k$ and for every vertex $v$, we have $\dic(T-v) = k-1$. In order words, the $k$-$\dic$-critical tournaments are precisely the minimal tournaments with dichromatic number at least $k$. Similarly, we say that $T$ is \emph{$k$-$\diomega$-critical} if $\diomega(T) = k$ and for every vertex $v$, we have $\diomega(T-v) = k-1$. Clearly, a tournament has directed clique number at least $k$ if and only if it contains a $k$-$\diomega$-critical tournament. If the number of $k$-$\diomega$-critical tournament was finite, then checking if a tournament $T$ contains any of them would be a polynomial-time algorithm that decides whether $\diomega(T) \ge k$ or not. For $k \ge 4$, this problem is NP-hard, so there are infinitely many $k$-$\diomega$-critical tournaments unless P=NP.

\medskip

On the other hand, it is easy to show that the one-vertex tournament is the only 1-$\diomega$-critical tournament, and that the directed triangle is the only 2-$\diomega$-critical tournament. We next present constructions of two infinite families of 3-$\diomega$-critical and 4-$\diomega$-critical tournaments respectively. Recently, using LLMs, Lelarge found an infinite family of 5-$\diomega$-critical tournaments, and formalized the proof in Rocq \cite{Lelarge}.

For every integer $n \ge 1$, we denote $T_{2n+1}$ the tournament with vertex set $\{v_0, \dots, v_{2n}\}$ and such that for $0 \le i, j \le 2n$, there is an arc from $v_i$ to $v_j$ if and ony if $j-i \in \{1, \dots, n-1, n+1\}$ modulo $2n+1$. Neumann-Lara and Urrutia \cite{critical} proved that the tournament $T_{2n+1}$ is 3-$\dic$-critical for all $n \ge 3$. 

\begin{proposition}
    For every integer $n\ge3$, the tournament $T_{2n+1}$ is 3-$\diomega$-critical. 
\end{proposition}

\begin{proof}
    Let $\prec$ be an ordering of $T_{2n+1}$. The last vertex in that ordering, say $v_i$, is adjacent to $v_{i+1}, \dots, v_{i+n-1}$, and $v_{i+n+1}$ in the backedge graph $T^\prec$. Moreover, $v_{i+1}$, $v_{i+2}$, and $v_{i+n+1}$ form a directed triangle in $T$, so two of them must be adjacent in $T^\prec$, and together with $v_i$ they form a triangle in $T^\prec$. Hence $\diomega(T_{2n+1}) \ge 3$.

    Since $T_{2n+1}$  is $3$-$\dic$-critical, for every vertex $v \in V(T_{2n+1})$, we have $\diomega(T-v) \le \dic(T-v) \le 2$.
\end{proof}

We say that a digraph $D$ is \emph{perfect} if for every induced subgraph $H$ of $D$, we have $\dic(H) = \diomega(H)$. Clearly, every $2$-dicolourable digraph is perfect, and by what precedes, the tournament $T_{2n+1}$ is perfect for all $n \ge 3$. Are there other perfect tournaments? One can check that the tournaments $D_3$ and $A_3$ are not perfect, so it follows from \Cref{thrm:An_Dn} that the class of perfect tournaments has bounded directed clique number and dichromatic number.

\begin{proposition}
    For every integer $n \ge 3$, the tournament $\Delta(T_{2n+1}, T_{2n+1}, T_{2n+1})$ is $4$-$\diomega$-critical. 
\end{proposition}

\begin{proof}
    Let $n \ge 3$ be an integer and $T=\Delta(A,B,C)$, where $A$, $B$, and $C$ are all isomorphic to $T_{2n+1}$. Consider an ordering $\prec$ of $V(T)$, and without loss of generality, assume that its first vertex is some $a\in A$. Then, every vertex of $C$ is adjacent to $a$ in $T^{\prec}$. Since $\diomega(C)=3$, we get that $\omega(T^{\prec})\ge 4$, and it follows that $\diomega(T)\ge 4$.

    It remains to prove that $\diomega(T\setminus x)=3$ for every $x\in V(T)$. By symmetry, it suffices to show this when $x \in A$. Let $b \in B$ and $c \in C$ be two vertices, and let $(A_1,A_2)$, $(B_1,B_2)$, and $(C_1,C_2)$ be $2$-dicolorings of $A \sm x$, $B \sm b$, and $C \sm c$ respectively (recall that $T_{2n+1}$ is 3-$\dic$-critical). Consider the ordering $\prec$ of $V(T) \sm x$ defined by
    \[
    b \prec c \prec A_1 \prec B_1 \prec C_1 \prec A_2 \prec B_2 \prec C_2,
    \]
    where for each $i \in \{1,2\}$, the sets $A_i$, $B_i$, and $C_i$ are internally ordered according to a topological ordering, so that each induces a stable set in $T^{\prec}$. Now, observe that $T^{\prec}$ is a subgraph of the graph with vertex set $V(T) \sm x$ in which
    \begin{itemize}
        \item $b$ is complete to $A_1 \cup B_1 \cup A_2 \cup B_2$;
        \item $c$ is complete to $B_1 \cup C_1 \cup B_2 \cup C_2$;
        \item $A_1$ is complete to $C_1 \cup A_2 \cup C_2$;
        \item $B_1$ is complete to $A_2 \cup B_2$;
        \item $C_1$ is complete to $B_2 \cup C_2$; and
        \item $A_2$ is complete to $C_2$.
    \end{itemize}
    One readily checks that the above graph, and thus $(T \setminus x)^{\prec}$ are $K_4$-free. Hence $\diomega(T \sm x) = 3$.
\end{proof}

The above construction can be generalised to build $(k+1)$-$\diomega$-critical tournament from any $k$-$\dic$-critical tournament $T$ with $\diomega(T) = k$, but we don't know if such tournaments exist for $k \ge 4$. 

\section{Open Problems}\label{sec:open}

In this section, we gather open problems about the directed clique number.

\medskip

Deciding $\diomega(D) \le k$ is in P for $k=1$, and NP-complete for $k \ge 3$ even when restricted to tournaments . The case $k=2$ is open and equivalent to deciding if a digraph admits a triangle-free backedge graph. 

\begin{problem}
    What is the complexity of deciding $\diomega(T) \le 2$ for tournaments / digraphs?
\end{problem}

We have established the existence of polynomial-time approximation algorithms for the directed clique number of tournaments. It is not known whether there are such algorithms for digraphs in general.

\begin{problem}
    Are there polynomial-time approximation algorithms for the directed clique number of digraphs in general?
\end{problem}

Gutowski and Rams \cite{complexity} proved that deciding if a digraph $D$ verifies $\diomega(D) \le k$ is $\Sigma_2^P$-complete when $k$ is given in the input, and conjectured that the same holds when restricted to tournaments.

\begin{problem}[Gutowski \& Rams \cite{complexity}]
    Is it $\Sigma_2^P$-hard to decide, given a tournament $T$ and an integer $k$, if $\diomega(T) \le k$?
\end{problem}

Kim \cite{Kim} and Aboulker {\it et al.} \cite{AACL23} showed that $\chi$-boundedness of a class of tournament is preserved by closure under substitution. The latter asked whether it is true for classes of digraphs in general, and if it holds for polynomial-$\dic$-boundedness (like in the undirected case \cite{substitution}). This question is discussed in \cite{pod} and \cite{Crew}.

\begin{problem}[Aboulker, Aubian, Charbit, \& Lopes \cite{AACL23}]
    Does closure under substitution preserves polynomial-$\dic$-boundedness of classes of tournaments? Does it preserves $\dic$-boundedness of classes of digraphs in general?
\end{problem}

Theorem \ref{thm:twinwidth} states that classes of tournaments with bounded twin-width are $\dic$-bounded. We may ask if they are polynomially $\dic$-bounded (as in the undirected case \cite{bourneuf}), and if it is true for classes of digraphs in general.

\begin{problem}
    Are classes of tournaments with bounded twin-width polynomially $\dic$-bounded? Are classes of digraphs with bounded twin-width $\dic$-bounded?
\end{problem}

The tournament analogue of the \Gya-Sumner conjecture (a tournament is $\dic$-bounding if and only if it that admit a backedge graph which is a forest) does not hold. We propose Conjecture \ref{conj:GS2} as a possible characterisation of $\dic$-bounding tournaments.

\begin{problem}[Aboulker, Aubian, Charbit, \& Lopes \cite{AACL23}]
    What are the $\dic$-bounding tournaments?
\end{problem}

Kim \cite{Kim} asked which finite sets of tournaments $\cal H$ are such that $\Forb(\cal H)$ has bounded dichromatic number. This follows the characterisation of tournaments $H$ such that $\Forb(H)$ has bounded dichromatic number by Berger {\it et al.} \cite{heroes}. The case of pairs is further investigated in \cite{KimKim}.

\begin{problem}[Kim \cite{Kim}]
    For which pairs of tournaments $\{T_1,T_2\}$ does $\Forb(T_1,T_2)$ have bounded dichromatic number?
\end{problem}

For each  $k \ge 3$, the NP-hardness of deciding $\diomega(T) \le k$ implies that they are infinitely many $k$-$\diomega$-critical tournaments. However, we do not have an explicit description of such families for $k > 5$.

\begin{problem}[Aboulker, Aubian, Charbit, \& Lopes \cite{AACL23}]
    Can we construct infinite families of $k$-$\diomega$-critical tournaments for every integer $k \ge 3$?
\end{problem}

The strong perfect graph theorem \cite{strong_perfect} is one the most celebrated result in graph theory. It would be interesting to study this notion for tournaments, or digraphs in general. 

\begin{problem}
    What are the perfect tournaments / digraphs?
\end{problem}

A natural question is: how large can the directed clique number of a $n$-vertex tournament be? Aboulker {\it et al.} exposed arbitrarily large tournaments satisfying $\diomega(T) \ge |V(T)|^{1/3-o(1)}$, whereas Gutowski and Rams \cite{complexity} proved that all tournaments verify $\diomega(T) \le \sqrt{2|V(T)|}$.

\begin{problem}[Aboulker et al. \cite{pod}]
    What is the maximum directed clique number of a $n$-vertex tournament (asymptotically)?
\end{problem}

Lastly, a curious question is whether all tournaments have an ordering which is both an $\diomega$-ordering and a $\dic$-ordering. We found no counter-examples, nor reasons to think it is true.

\begin{problem}[Aboulker, Aubian, Charbit, \& Lopes \cite{AACL23}]
    Does every tournament $T$ have an ordering $\prec$ such that $\diomega(T) = \omega(T^\prec)$ and $\dic(T)=\chr(T^\prec)$?
\end{problem}

\paragraph*{Acknowledgement} We thank Pierre Aboulker and Pierre Charbit for many insightful discussions and helpful comments. They were present at all stages, and participated to some of the results, in particular in \Cref{sec:crit}.

\bibliographystyle{plain}
\bibliography{refs}

\end{document}